\documentclass{TPmod}

\usepackage{subfigure}
\usepackage{amsmath,amsthm}

\DeclareMathOperator{\HC}{\mathsf{HC}}
\DeclareMathOperator{\HP}{\mathsf{HP}}
\DeclareMathOperator{\tildeint}{\overset{\,\,\,\,\sim}{\int}}
\DeclareMathOperator{\fink}{\mathsf{fin}}

\newcommand{\vshs}{$\mathsf{VSHS}$}
\newcommand{\vhs}{$\mathsf{VHS}$}
\newcommand{\power}[1]{[\![ #1 ]\!]}
\newcommand{\laurents}[1]{(\!( #1 )\!)}
\def\cM{\mathcal{M}}

\newtheorem{nota}[thm]{Notation}

\begin{document}

\title{Formulae in noncommutative Hodge theory}

\author[Nick Sheridan]{Nick Sheridan}

\address{Nick Sheridan, School of Mathematics, University of Edinburgh, Edinburgh EH9 3FD, UK}

\begin{abstract}
{\sc Abstract:} 
We prove that the cyclic homology of a saturated $A_\infty$ category admits the structure of a `polarized variation of Hodge structures', building heavily on the work of many authors: the main point of the paper is to present complete proofs, and also explicit formulae for all of the relevant structures. 
This forms part of a project of Ganatra, Perutz and the author, to prove that homological mirror symmetry implies enumerative mirror symmetry.
 \end{abstract}

\maketitle

\section{Introduction}

\subsection{Calabi--Yau mirror symmetry}

Mirror symmetry predicts the existence of certain `mirror' pairs of Calabi--Yau K\"{a}hler manifolds, $X$ and $Y$, so that the Gromov--Witten invariants of $X$ can be extracted from certain Hodge-theoretic invariants of $Y$.\footnote{We only consider genus-zero Gromov--Witten invariants in this paper.}
The first thrilling application of mirror symmetry was the prediction of the number of rational curves, in all degrees, on the quintic threefold $X$, in terms of the Hodge theory of a mirror manifold $Y$ \cite{Candelas1991}. 
This prediction for the quintic, together with many more examples of mirror symmetry, was later mathematically verified \cite{Givental1996,Lian1997}. 

The most conceptually satisfying way of formulating the mirror relationship between numerical invariants of $X$ and $Y$ is as an isomorphism of \emph{variations of Hodge structures} (\vhs) \cite{Morrison1993}. 
A \vhs{} over a complex manifold $\cM$ consists of a holomorphic vector bundle $V \to \cM$, equipped with a filtration $F^{\ge *}V$ by holomorphic subbundles, and a flat connection $\nabla$ satisfying the condition $\nabla_v F^{\ge i}V \subset F^{\ge i-1} V$ known as \emph{Griffiths transverality}.\footnote{This is more precisely referred to as a \emph{complex variation of Hodge structures}.}
We introduce the \emph{K\"ahler moduli space} $\cM_{K\ddot{a}h}(X)$, which parametrizes deformations of the K\"ahler form on $X$, and the \emph{complex moduli space} $\cM_{cpx}(Y)$, which parametrizes deformations of the complex structure on $Y$. 
The Gromov--Witten theory of $X$ gets packaged into the \emph{$A$-model \vhs}, which lives over the K\"ahler moduli space: $V^A(X) \to \mathcal{M}_{K\ddot{a}h}(X)$. 
The Hodge theory on deformations of $Y$ gets packaged into the \emph{$B$-model \vhs}, which lives over the complex moduli space: $V^B(Y) \to \cM_{cpx}(Y)$.
\emph{Hodge-theoretic} mirror symmetry then predicts an isomorphism $V^A(X) \cong V^B(Y)$, covering an isomorphism $\cM_{K\ddot{a}h}(X) \cong \cM_{cpx}(Y)$ called the \emph{mirror map}. 
Enumerative mirror symmetry, i.e., the explicit formulae relating the numerical invariants, can be deduced from this isomorphism of \vhs{} (see, e.g., \cite{coxkatz}).

Kontsevich proposed a generalization of Hodge-theoretic mirror symmetry called \emph{homological} mirror symmetry \cite{Kontsevich1994}. 
It predicts a quasi-equivalence of $A_\infty$ categories $ \EuF(X) \simeq \EuD(Y)$, where $\EuF(X)$ is the split-closed derived Fukaya category of $X$, and $\EuD(Y)$ is a $\mathsf{dg}$ enhancement of the bounded derived category of coherent sheaves on $Y$. 
More precisely, one should think of these as families of categories, parametrized by the K\"ahler and complex moduli spaces respectively; and the quasi-equivalence matches the Fukaya category living over a point in K\"ahler moduli space with the derived category living over the corresponding point in complex moduli space, where the correspondence between moduli spaces is given by the same mirror map as before.

Kontsevich also predicted that homological should imply Hodge-theoretic mirror symmetry. 
This was subsequently made more precise \cite{Barannikov2002,Katzarkov2008,Costello2009}. 
The expectation is that the cyclic homology of a family of saturated $A_\infty$ categories carries the structure of a \vhs{}; the cyclic homology of the Fukaya category is isomorphic to $V^A(X)$; and the cyclic homology of the bounded derived category of coherent sheaves is isomorphic to $V^B(Y)$.
Therefore, the equivalence of categories implies the isomorphism of \vhs{}.

This paper forms part of a project of the author, joint with Ganatra and Perutz, to carry out this program.  
Theorem \ref{thm:main} implies roughly that the cyclic homology of a family of saturated $\Z$-graded $A_\infty$ categories carries the structure of a \vhs, and this structure is functorial under $A_\infty$ functors; we also give explicit formulae for all of the relevant structures. 
Ganatra \cite{Ganatra:OCcyc} has defined a map from the cyclic homology of the Fukaya category to the $A$-model \vhs; this map is shown to respect the \vhs{} structure in \cite{Ganatra:OCconn}, using the formulae in the present paper; and the map is shown to be an isomorphism in \cite{Ganatra2015}. 
The corresponding comparison theorem for the $B$-model is known to experts, although not everything is written in the literature; modulo this $B$-model comparison theorem, the proof is complete. 
  
We refer to \cite{Ganatra2015} for precise statements of the results. 
One corollary of them is a new proof of the mirror symmetry predictions for Gromov--Witten invariants of the quintic, as a consequence of the proof of homological mirror symmetry for the quintic \cite{Sheridan2015}. 
The $B$-model comparison theorem has been established in this specific case by Tu \cite{Tu2019}.

\subsection{Fano mirror symmetry} 

Although mirror symmetry was originally formulated for mirror pairs of Calabi--Yau K\"ahler manifolds $(X,Y)$, it admits a generalization in which $X$ is allowed to be Fano. 
In this case the mirror is no longer a manifold, but rather a `Landau--Ginzburg model' $(Y,W)$, which means a variety $Y$ equipped with a function $W:Y \to \C$. 
The Gromov--Witten invariants of $X$ are related to the singularity theory of $W$ \cite{Givental1995}.

Once again, the relation between the numerical invariants can be expressed in terms of an isomorphism $V^A(X) \cong V^B(Y,W)$; however, the structures getting identified are now variations of \emph{semi-infinite} Hodge structures (\vshs). 
This notion was introduced by Barannikov \cite{Barannikov2001}, but the study of the $B$-model \vshs{} associated to $(Y,W)$ goes back to Saito \cite{Saito1983}.

In Section \ref{sec:vhs} we define the notion of a \emph{graded} \vshs. The following notions are equivalent:
\begin{align*}
\text{$\Z/2$-graded \vshs} \quad&\longleftrightarrow\quad \text{\vshs{} in Barannikov's sense;} \\
\text{$\Z$-graded \vshs}\quad& \longleftrightarrow \quad\text{\vhs{} in the sense of the previous section.}
\end{align*}
We refer to \cite[Lemma 2.7]{Ganatra2015} for a precise statement and proof of the latter equivalence, which is a version of the `Rees correspondence' between filtered bundles over $\cM$ and equivariant bundles over $\cM \times \mathbb{A}^1$ \cite{Simpson:HF}. 
Thus, Hodge-theoretic mirror symmetry can always be formulated as an isomorphism of graded \vshs{}; in the Fano case the grading group is $\Z/2$, and in the Calabi--Yau case it is $\Z$. 
For the rest of the paper, `\vshs' will always mean `graded \vshs', with the grading group implicit.

Homological mirror symmetry admits a generalization to the Fano case: roughly, it predicts a quasi-equivalence of $\Z/2$-graded $A_\infty$ categories $\EuD\EuF(X) \simeq MF(Y,W)$, where the latter is the category of \emph{matrix factorizations} of $W$ \cite{orlov04}.  
One expects that it should imply Hodge-theoretic mirror symmetry, by a similar argument to the Calabi--Yau case, but with \vhs{} replaced by \vshs.
 
Our Theorem \ref{thm:main} implies roughly that the cyclic homology of a family of saturated $Y$-graded $A_\infty$ categories, which satisfies the degeneration property, carries the structure of a $Y$-graded \vshs. 
The case $Y=\Z/2$ is the one relevant to Fano mirror symmetry, and the case $Y=\Z$ is the one relevant to Calabi--Yau mirror symmetry. 
In the latter case the degeneration property holds automatically, by Kaledin's proof \cite{Kaledin2017} of Kontsevich--Soibelman's degeneration conjecture \cite[Conjecture 9.1.2]{Kontsevich2006a}. 
In particular, our result implies that the cyclic homology of a family of saturated $\Z$-graded $A_\infty$ categories carries the structure of a $\Z$-graded \vshs{}, which we recall is equivalent to a \vhs{} (we stated this version of the result in the previous section).
We remark that the grading group does not enter into the proof of Theorem \ref{thm:main}: the cases of relevance to Fano and  Calabi--Yau mirror symmetry are handled uniformly.

\begin{rmk}
It should be possible to prove some version of the statement that homological implies Hodge-theoretic mirror symmetry in the Fano case, following the argument in the Calabi--Yau case (and in particular, using the $\Z/2$-graded case of our Theorem \ref{thm:main}).
See \cite{AmorimTu} for some recent progress.
\end{rmk}

\subsection{Standing notation}

Let $\Bbbk \subset \BbK$ be fields.
We will write $\cM := \spec \BbK$, and $T\mathcal{M}:=\deriv_\Bbbk \BbK$. 

We fix a grading group throughout: more precisely, we fix a `grading datum' in the sense of \cite{Sheridan2015}, which is an abelian group $Y$ together with homomorphisms $\Z \to Y \to \Z/2$ whose composition is non-zero.
All of our structures are $Y$-graded; when we talk about an element of degree $k \in \Z$, we really mean its degree is the image of $k$ under the map $\Z \to Y$;  and when we write a Koszul-type sign $(-1)^{|a|}$, it means that the image of the $Y$-degree of $a$, under the map $Y \to \Z/2$, is $|a| \in \Z/2$. 

\begin{rmk}
The two most relevant grading data are $\Z:=\{\Z \xrightarrow{\id} \Z \to \Z/2\}$ and $\Z/2 :=\{\Z \to \Z/2 \xrightarrow{\id} \Z/2\}$; working with the former is equivalent to working with ordinary $\Z$-gradings, while working with the latter is equivalent to working with ordinary $\Z/2$-gradings. 
\end{rmk}

We define $\BbK \power{u}$ to be the graded ring of formal power series in a formal variable $u$ of degree $2$, and $\BbK\laurents{u}$ the graded ring of formal Laurent series. 
For any $f \in \BbK\power{u}$ or $\BbK \laurents{u}$, we denote
\[ f^\star(u) := f(-u).\]

\begin{rmk} 
To be precise, $\BbK\power{u}$ (respectively, $\BbK\laurents{u}$) is the degreewise completion of $\BbK[u]$ (respectively, $\BbK[u,u^{-1}]$) with respect to the $u$-adic filtration. 
If the morphism $\Z \to Y$ is not injective then this includes infinite sums of powers of $u$, but if the morphism is injective then the completion has no effect because all powers of $u$ have different degrees. 
Thus $\BbK\laurents{u}$ contains `semi-infinite' sums of powers of $u$ in the $\Z/2$-graded case; hence Barannikov's terminology.
\end{rmk}

Finally, if $\sigma \in \Z/2$, we denote $\sigma' := \sigma - 1$.

\subsection{Main result}

We define various flavours of \vshs{} in Section \ref{sec:vhs}. 
To give an idea of what they mean, let us explain roughly what they correspond to under the Rees correspondence, in the $\Z$-graded case:
\begin{itemize}
\item An \emph{unpolarized pre-\vshs} over $\cM$ corresponds to an $\mathcal{O}_\cM$-module equipped with a filtration $F^{\ge *}$ and flat connection $\nabla$ satisfying Griffiths transversality. 
\item An \emph{unpolarized \vshs} is an unpolarized pre-\vshs{} such that the $\mathcal{O}_\cM$-module is a finite-rank vector bundle. 
\item A \emph{polarization} for an unpolarized pre-\vshs{} is a covariantly constant pairing $(\cdot,\cdot)$ such that $(F^{\ge j},F^{\ge k}) = 0$ for $j+k <0$, and $(a,b) = (-1)^n(b,a)$ for some $n \in \Z/2$ called the \emph{weight}; a \emph{polarized pre-\vshs} is a pre-\vshs{} equipped with a polarization. 
\item A \emph{polarized \vshs} is a polarized pre-\vshs, such that the $\mathcal{O}_\cM$-module is a finite-rank vector bundle, and the pairing is non-degenerate.
\end{itemize}

\begin{rmk}
Note that a polarized/unpolarized \vshs{} is the same thing as a polarized/unpolarized pre-\vshs{} satisfying additional properties, rather than equipped with additional data. 
\end{rmk}

Our main results concern the Hochschild invariants of $A_\infty$ categories:

\begin{main}\label{thm:main}
Let $\EuC$ be a $\BbK$-linear graded $A_\infty$ category. Then:
\begin{enumerate}
\item \label{it:unpprevshs} Its negative cyclic homology $\HC_\bullet^-(\EuC)$, endowed with the Getzler--Gauss--Manin connection $\nabla$  \cite{Getzler1993}, carries the structure of an unpolarized pre-\vshs{} over $\cM$. 
\item \label{it:polprevshs} If $\EuC$ is proper and admits an $n$-dimensional weak proper Calabi--Yau structure (see Section \ref{subsec:sym} for the definition), then $(\HC_\bullet^-(\EuC),\nabla)$ admits a natural polarization $\langle \cdot, \cdot \rangle_{res}$  of weight $n$, given by Shklyarov's higher residue pairing \cite{Shklyarov2013}. 
\item \label{it:mor} These structures are Morita invariant.
\item \label{it:degen} If $\EuC$ is saturated and its noncommutative Hodge-to-de Rham spectral sequence degenerates, then the polarized pre-\vshs{} $(\HC_\bullet^-(\EuC),\nabla,\langle \cdot,\cdot\rangle_{res})$ is in fact a polarized \vshs. 
\end{enumerate}
\end{main}

\begin{rmk}
A conjecture of Kontsevich and Soibelmann \cite[Conjecture 9.1.2]{Kontsevich2006a} says that the noncommutative Hodge-to-de Rham spectral sequence degenerates for any saturated $\EuC$. 
The conjecture has been proved by Kaledin in the case that $\EuC$ is $\Z$-graded \cite{Kaledin2017} (see also \cite{Mathew2017}); so far as the author is aware it remains open if $\EuC$ is, for example, $\Z/2$-graded.
\end{rmk}

\begin{rmk}
Katzarkov--Kontsevich--Pantev conjecture that the \vshs{} of Theorem \ref{thm:main} can be endowed with a natural $\Q$-structure (see \cite[Section 2.2.6]{Katzarkov2008}, and also \cite{Blanc2016}).
\end{rmk}

Let us comment on the originality of Theorem \ref{thm:main}. 
We believe our contribution ranges from `writing down explicit formulae for known structures with uniform sign conventions, as a handy reference' at the low end, to `checking that these structures have certain natural (but slightly tricky-to-prove) compatibilities' at the high end. 
To be precise: 
\begin{itemize}
\item Our proof of the Morita invariance of the Getzler--Gauss--Manin connection is new. 
\item Shklyarov's construction of the higher residue pairing for $\mathsf{dg}${}  categories \cite{Shklyarov2013} immediately gives the construction for $A_\infty$ categories, because any $A_\infty$ category is quasi-equivalent to a $\mathsf{dg}${}  category via the Yoneda embedding. On the other hand, for an $A_\infty$ category whose morphism spaces are finite-dimensional on the chain level, work of Costello \cite{Costello2007} and Konstevich--Soibelman \cite{Kontsevich2006a} implies that there should be an explicit formula for the pairing. We write down the formula and prove that it is equivalent to Shklyarov's definition (see Proposition \ref{prop:higherresform}). 
Our proof is motivated by Shklyarov's \cite[Proposition 2.6]{Shklyarov2013}. 
\item Using the explicit formula for the higher residue pairing, we establish that it is covariantly constant with respect to the Getzler--Gauss--Manin connection: we believe that this result is also new (a related result was proven by Shklyarov in \cite{Shklyarov2013}, but that was for a different version of Getzler's connection, namely the one in the $u$-direction rather than in the direction of the base).
\end{itemize}  

Now we give a guide, to help the reader find the proofs of the different parts of Theorem \ref{thm:main}. 
Part \eqref{it:unpprevshs} is proved in Section \ref{sec:hoch}. 
Part \eqref{it:polprevshs} is proved in two parts: covariant constancy of the higher residue pairing, with respect to the Getzler--Gauss--Manin connection, is proved in Corollary \ref{cor:hrescov}; and symmetry of the higher residue pairing (which is the only part that requires the weak proper Calabi--Yau structure) is proved in Lemma \ref{lem:nsym}. 
The tools to prove part \eqref{it:mor} are developed in Section \ref{sec:mor}; Morita invariance of negative cyclic homology and the Getzler--Gauss--Manin connection are proved in Corollary \ref{cor:cycGMmor}, and Morita invariance of the higher residue pairing is proved in Proposition \ref{prop:higherresMor}. 
Part \eqref{it:degen} is proved in Theorem \ref{thm:degen}.

The paper involves a lot of long formulae composing multilinear operations in complicated ways, with non-trivial sign factors. 
We explain a graphical notation for these composition rules in Appendix \ref{app:signs}, which allows one to check various identities, with signs, in an efficient way; and we draw the graphical notation for some of the trickiest signs that appear in the paper. 
We omit the proofs of some identities that become trivial using the graphical notation, however we have tried very hard to write down explicitly the correct signs for every operation we define.

\paragraph{Acknowledgments.}
I would like to thank Sheel Ganatra and Tim Perutz, my collaborators on the larger project of which this paper is a part (see \cite{Ganatra2015,Perutz2015,Perutz2015a,Ganatra:OCconn}).
The material in this paper was originally intended to form a background section to one of the papers in that series, but it turned out to be so long, and in such a different direction, that it made sense to split it off. 
Conversations with Ganatra and Perutz, and their suggestions, were extremely helpful in completing the paper. I am also very grateful to Paul Seidel for helpful conversations about the Mukai pairing. 
I thank Lino Amorim and Junwu Tu for pointing out some sign errors in a previous version. 
I thank the anonymous referee for helpful suggestions. 
While working on this project I was partially supported by the National Science Foundation through Grant number
DMS-1310604, and under agreement number DMS-1128155. 
Any opinions, findings and conclusions or recommendations expressed in this material are those of the author and do not necessarily reflect the views of the National Science Foundation.
I was also partially supported by a Royal Society University Research Fellowship. 
I am also grateful to the IAS and the Instituto Superior T\'{e}cnico for hospitality while working on this project.

\section{Variations of semi-infinite Hodge structures: definitions}
\label{sec:vhs}

Variations of semi-infinite Hodge structures (\vshs) were introduced in \cite{Barannikov2001}. 
Here we recall the basic definitions, following \cite[Section 2.2]{Coates2009} and \cite[Chapter 2]{Gross2011}. 
We break with certain conventions in the literature, for which we apologize. 
We point out the places where our conventions differ as we go along.

Recall our standing notation: we fix a grading datum $\Z \to Y \to \Z/2$, and fields $\Bbbk \subset \BbK$, and define $\mathcal{M} := \spec \BbK$, $T\mathcal{M} := \deriv_\Bbbk \BbK$.

\subsection{Pre-\vshs}
\label{subsec:vhsdef}

\begin{defn}
\label{defn:prevhsnopair}
An \emph{unpolarized pre-\vshs} over $\mathcal{M}$ consists of a graded $\BbK\power{u}$-module $\EuE$, equipped with a flat connection\footnote{More precisely, there is a map $u\nabla: T\mathcal{M} \otimes \EuE \to \EuE$, such that $u\nabla_X s$ is $\BbK$-linear in $X$, additive in $s$, satisfying
\begin{align}
\tag{\emph{Leibniz rule}} u\nabla_X(f \cdot s) &= uX(f) \cdot s + f \cdot u\nabla_X s\quad \text{for $f \in \BbK \power{u}$, and}\\
\tag{\emph{flatness}} [u\nabla_X, u\nabla_Y] &= u^2\nabla_{[X,Y]} \quad \text{for all $X,Y \in T\mathcal{M}$.}
\end{align}}  $\nabla: T\mathcal{M} \otimes_\BbK \EuE \to u^{-1} \EuE$ of degree $0$.
\end{defn}

\begin{defn}
\label{defn:prevhspair}
A \emph{polarization} for a pre-\vshs{} is a pairing
\[ (\cdot,\cdot):\EuE \times \EuE  \to \BbK\power{u}\]
of degree $0$, additive in both inputs, and satisfying
\begin{align}
\tag{\emph{sesquilinearity}} (f\cdot s_1,g \cdot s_2) &= f\cdot g^\star \cdot (s_1,s_2) \quad \text{ for $f,g \in \BbK \power{u}$,} \\
\tag{\emph{covariant constance}} uX(s_1,s_2) &= (u\nabla_X s_1,s_2) - (s_1,u\nabla_X s_2), \quad \text{and} \\
\tag{\emph{symmetry}}(s_1,s_2) &= (-1)^{n+|s_1|\cdot|s_2|}(s_2,s_1)^\star, 
\end{align}
where $n \in \Z/2$ is called the \emph{weight}.
\end{defn}

\begin{rmk}
It is not usually assumed that a polarization must have degree $0$: we prefer to shift whatever pre-\vshs{} we are considering so that this is the case. 
In this paper, polarizations will arise from Shklyarov's higher residue pairing on cyclic homology, which has degree $0$ with respect to the standard grading.
\end{rmk}

\begin{defn}
\label{defn:vhsmorph}
A \emph{morphism of pre-\vshs} is a degree-$0$ morphism of $\BbK\power{u}$-modules $F: \EuE_1 \to \EuE_2$ which respects the connections, and, if the pre-\vshs{} are polarized, satisfies $(F(\alpha),F(\beta))_2 = (\alpha,\beta)_1$.
\end{defn}

\subsection{\vshs}

\begin{defn}
\label{defn:vhs}
An \emph{unpolarized \vshs} is an unpolarized pre-\vshs{} such that $\EuE$ is a free $\BbK\power{u}$-module of finite rank.
\end{defn}

\begin{defn}
\label{defn:polvhs}
A \emph{polarization} for a \vshs{} is a polarization for the underlying pre-\vshs, which is furthermore \emph{non-degenerate}: i.e., the pairing of $\BbK$-vector spaces
\[ \EuE /u\EuE  \otimes_\BbK \EuE /u\EuE  \to \BbK\]
induced by $(\cdot,\cdot)$ is non-degenerate.
\end{defn}

A morphism of \vshs{} (polarized or unpolarized) is the same thing as a morphism of the underlying pre-\vshs.

\begin{rmk}
What we call a `$\Z/2$-graded polarized \vshs' is usually simply called a \vshs{}, in particular, the polarization is part of the structure. 
However, the notion of an unpolarized \vshs{} has applications in mirror symmetry so it seems useful to make the distinction. 
\end{rmk}

\begin{rmk}
An unpolarized $\Z$-graded \vshs{} is equivalent to a vector bundle over the $\Bbbk$-scheme $\cM$, equipped with a filtration and flat connection satisfying Griffiths transversality \cite[Lemma 2.7]{Ganatra2015}; in the application to mirror symmetry we take $\Bbbk=\C$. 
The most relevant choice of $\BbK$ for the application to mirror symmetry is $\BbK = \C \laurent{q}$. 
One can think of $\cM = \spec \BbK$ as a `formal punctured disc'. 
Geometrically, one thinks of this formal punctured disc as mapping into the K\"ahler/complex moduli space, with the limit $q \to 0$ corresponding to a `large volume'/`large complex structure' limit point of the moduli space. 
We consider the pullback of the relevant structures (\vshs, categories) to $\cM$. 
We shall define a `family of $A_\infty$ categories parametrized by $\cM$' to be a $\BbK$-linear $A_\infty$ category. 
Thus Theorem \ref{thm:main} constructs a \vshs{} over $\cM$ from a family of $A_\infty$ categories parametrized by $\cM$.
\end{rmk}

\subsection{Euler gradings}

We have assumed that our \vshs{} are graded, in the sense that $\EuE$ is a direct sum of its graded pieces. 
A different notion of `graded \vshs{}' is used in the literature \cite{Barannikov2001,Coates2009}, which we will instead refer to as an `Euler-graded \vshs'. 
We depart from the standard terminology in this way to avoid confusion regarding the usual terminology for $A_\infty$ categories. 
Namely, if $\EuC$ is a \emph{graded} $A_\infty$ category (in the usual sense), which is saturated and whose noncommutative Hodge-to-de Rham spectral sequence degenerates, then Theorem \ref{thm:main} constructs a \emph{graded} \vshs{} in our sense; it does not construct an Euler-graded \vshs{}. 

\begin{defn}
\label{defn:vhseulgrad}
An \emph{Euler grading} on a pre-\vshs{} is a $\Bbbk$-linear endomorphism $\mathsf{Gr}: \EuE \to \EuE$, such that there is a vector field $E \in T\mathcal{M}$ (called the \emph{Euler vector field}) satisfying
\begin{align*}
 \mathsf{Gr}(f \cdot s) &= (2u \partial_u + 2E)f \cdot s + f \cdot \mathsf{Gr}(s) \quad \text{for all $f \in \BbK \power{u}$,}\\
 [ \mathsf{Gr}, \nabla_X] &= \nabla_{[2E,X]},\quad\text{and}\\
 (2u \partial_u + 2E) (s_1,s_2) &= (\mathsf{Gr}(s_1),s_2) + (s_1,\mathsf{Gr}(s_2)).
 \end{align*}
A morphism of Euler-graded pre-\vshs{} is required to satisfy $\mathsf{Gr}_2 \circ F = F \circ \mathsf{Gr}_1$.
\end{defn}

\begin{example}
An $\R$-graded pre-\vshs{} admits an Euler grading, by setting $\mathsf{Gr}(s)  :=  |s| \cdot s$. 
The Euler vector field $E \in T\mathcal{M}$ is the graded derivation $E: \BbK \to \BbK$ defined by the same formula as $\mathsf{Gr}$, multiplied by $\frac{1}{2}$.
\end{example}

\begin{rmk}
For any Euler-graded pre-\vshs{}, we can extend the connection $\nabla$ to a flat connection that is also defined in the $u$-direction, by setting $\nabla_{u\frac{\partial}{\partial u}} := \frac{1}{2}\mathsf{Gr} - \nabla_{E}$.
\end{rmk}

\begin{rmk}
One can also extract an Euler graded \vshs{} from an $A_\infty$ category $\EuC$, if one assumes that the category itself comes with an `Euler grading'. 
Namely, one assumes that there is an Euler vector field $E$ on the coefficient field $\BbK$, and that $\EuC$ is a $\Z/2$-graded $\BbK$-linear $A_\infty$ category. 
Then an Euler grading on $\EuC$ is a map $\mathsf{Gr}$ on the morphism spaces of the category, compatible with the Euler vector field as in Definition \ref{defn:vhseulgrad}, and such that the $A_\infty$ structure maps $\mu^s$ satisfy
\[ \mathsf{Gr} \circ \mu^s = \mu^s \circ \mathsf{Gr} + (2-s) \cdot \mu^s.\]
Our proof of Theorem \ref{thm:main} applies to the $\Z/2$-graded $A_\infty$ category $\EuC$, to produce $\Z/2$-graded Hochschild invariants: and one easily checks that, if $\EuC$ comes with an Euler grading, then all of the Hochschild invariants admit Euler gradings, compatible with all structures. 
This is relevant when one studies mirror symmetry for Fano varieties, but not in the Calabi--Yau case. 
We will not comment further on it in this paper.
\end{rmk}

\section{Hochschild invariants of $A_\infty$ categories}\label{sec:hoch}

\subsection{$\mathsf{dg}$ categories}

\begin{defn}
A \emph{$\BbK$-linear $\mathsf{dg}$ category} $\EuC$ consists of: a set of objects $Ob(\EuC)$; for each pair of objects, a graded $\BbK$-vector space $hom^\bullet(X,Y)$, equipped with a differential $d$ of degree $+1$; composition maps
\begin{align*}
hom^\bullet(X_1,X_2) \otimes hom^\bullet(X_0,X_1) &\to hom^\bullet(X_0,X_2)
\end{align*}
of degree $0$, which we denote by $f \otimes g  \mapsto  f \cdot g$, satisfying
\begin{align}
\label{eqn:assoc}
(f \cdot g) \cdot h &= f \cdot (g \cdot h), \\
\label{eqn:leibniz}
d(f \cdot g) &= df \cdot g + (-1)^{|f|} f \cdot dg,
\end{align}
and there exists a unit $e_X \in hom^0(X,X)$ for all $X$, satisfying $de_X=0$ and
\begin{equation}
\label{eqn:dgunit}
f \cdot e_X = f = e_Y \cdot f \quad \text{for all $f \in hom^\bullet(X,Y)$. }\end{equation}
\end{defn}

\begin{example}
Let $R$ be a graded $\BbK$-algebra. 
There is a $\mathsf{dg}$ category $\modules R$ whose objects are cochain complexes $\{ \ldots \to M_p \overset{d_M}{\to} M_{p+1} \to \ldots \}$ of $R$-modules; it is defined by
\begin{align*}
 hom^p_{\modules R}(M_\bullet,N_\bullet) &:= \bigoplus_j \Hom_R(M_j,N_{j+p}),\\
 df &:= d_N \circ f + (-1)^{|f|'} f \circ d_M,\\
 f \cdot g &:= f \circ g.
 \end{align*}
\end{example}

\begin{defn}
\label{defn:oppdg}
Let $\EuC$ be a $\mathsf{dg}$ category; we define the \emph{opposite $\mathsf{dg}$ category} $\EuC^{op}$ with the same set of objects, by setting
\[
hom^\bullet_{\EuC^{op}}(X,Y)  :=  hom^\bullet_{\EuC}(Y,X); \quad
d_{op}(x) :=  d(x); \quad
f \cdot_{op} g  :=  (-1)^{|f| \cdot |g|} g \cdot f.\]
\end{defn}

\subsection{$A_\infty$ categories}
\label{s:ainf}

We follow the sign conventions of \cite{Getzler1993} and \cite{Seidel2008c}. 
A pre-$A_\infty$ category $\EuC$ consists of a set of objects, and a graded $\BbK$-vector space $hom^\bullet_\EuC(X,Y)$ for each pair of objects $X$, $Y$. 

We define the convenient notation
\[ \EuC(X_0,\ldots,X_s) := hom^\bullet(X_0,X_1)[1] \otimes \ldots \otimes hom^\bullet(X_{s-1},X_s)[1].\]
For a generator $a_1 \otimes \ldots \otimes a_s$ of $\EuC(X_0,\ldots,X_s)$, we define 
\[ \epsilon_j := | a_1 |' + \ldots + |a_j |' \in \Z/2.\]
We define the Hochschild cochains of length $s$:
\[ CC^\bullet(\EuC)^s := \prod_{X_0,\ldots,X_s} \Hom^\bullet(\EuC(X_0,\ldots,X_s),\EuC(X_0,X_s))[-1].\]
We then define the \emph{Hochschild cochain complex}
\[ CC^\bullet(\EuC) := \prod_{s \ge 0} CC^\bullet(\EuC)^s\]
(more precisely, the completion of the direct sum in the category of graded vector spaces, with respect to the filtration by length $s$).
It admits the \emph{Gerstenhaber product}:
\[ \varphi \circ \psi(a_1,\ldots,a_s) := \sum_{j,k} (-1)^{|\psi|' \cdot \epsilon_j} \varphi^*(a_1,\ldots,\psi^*(a_{j+1},\ldots),a_{k+1},\ldots).\]
An $A_\infty$ structure on $\EuC$ is an element $\mu^* \in CC^2(\EuC)$ satisfying $\mu^* \circ \mu^* = 0$ and $\mu^0 = 0$. 

The \emph{cohomology category} is a graded $\BbK$-linear category $H^\bullet(\EuC)$ with the same objects, 
\begin{align*} 
\mathrm{Hom}^\bullet(X,Y) &:= H^\bullet(hom^\bullet(X,Y),\mu^1),\\
 [a_1] \cdot [a_2] &= (-1)^{|a_2|} \mu^2(a_2,a_1).
 \end{align*}
We will assume that our $A_\infty$ categories are \emph{cohomologically unital}, i.e., $\mathrm{Hom}^\bullet(X,Y)$ admits units.

We recall that an \emph{$A_\infty$ functor} $F: \EuC \to \EuD$ consists of a map on the level of objects, together with maps
\[ F^s: \EuC(X_0,\ldots,X_s) \to \EuD(FX_0,FX_s)\]
satisfying
\begin{multline}
\label{eqn:Ainfun}
\sum \mu^*(F^*(a_1,\ldots),F^*(a_{j_1+1},\ldots),\ldots,F^*(a_{j_k+1},\ldots,a_s)) = \\
\sum (-1)^{\epsilon_j} F^*(a_1,\ldots,\mu^*(a_{j+1},\ldots),\ldots,a_s).
\end{multline}

\begin{defn}
\label{defn:ainfdg}
If $\EuC$ is a $\mathsf{dg}$ category, we define an $A_\infty$ category $A_\infty(\EuC)$ with the same set of objects:
\begin{align*}
hom^\bullet_{A_\infty(\EuC)}(X,Y) &:= hom^\bullet_\EuC(Y,X);\\
\mu^1(f)  :=  df; \qquad \mu^2(f,g) & :=  (-1)^{|f|} f \cdot g;\qquad \mu^{\ge 3}  :=  0.
\end{align*}
If $F: \EuC \to \EuD$ is a $\mathsf{dg}$ functor, we define an $A_\infty$ functor $A_\infty(F): A_\infty(\EuC) \to A_\infty(\EuD)$ by setting $A_\infty(F)^1 := F$ and $A_\infty(F)^{\ge 2} := 0$.
\end{defn}

\begin{defn}
\label{defn:oppainf}
If $\EuC$ is an $A_\infty$ category, we define the \emph{opposite} $A_\infty$ category $\EuC^{op}$. 
It has the same objects as $\EuC$, and morphism spaces $hom^\bullet_{\EuC^{op}}(X,Y) := hom^\bullet_\EuC(Y,X)$.
We have an isomorphism $CC^\bullet(\EuC)  \to CC^\bullet(\EuC^{op})$ sending $\eta  \mapsto \eta_{op}$, defined by
\begin{align*}
 \eta^s_{op}(a_1,\ldots,a_s) &:= (-1)^\dagger \eta^s(a_s,\ldots,a_1),\quad\text{where}\\
 \dagger &:= \sum_{1 \le i<j \le s} |a_i|' \cdot |a_j|'.
 \end{align*}
This isomorphism preserves the Gerstenhaber product, in the sense that $\alpha_{op} \circ \beta_{op} = (\alpha \circ \beta)_{op}$. 
Thus we can define the $A_\infty$ structure maps on $\EuC^{op}$ to be equal to $\mu^*_{op}$.
\end{defn}

\begin{rmk}
\label{rmk:oppcohunit}
If $e \in hom^0_\EuC(X,X)$ is a cohomological unit in $\EuC$, then $-e \in hom^0_{\EuC^{op}}(X,X) = hom^0_\EuC(X,X)$ is a cohomological unit in $\EuC^{op}$.
\end{rmk}

\begin{rmk}
This is a different definition of the opposite $A_\infty$ category from that given in \cite[Section 1a]{Seidel2008} (i.e., the two definitions give non-equivalent $A_\infty$ categories in general). 
It was verified in \cite[Appendix B]{Sheridan2017} that there is an isomorphism $\EuF(X,\omega)^{op} \cong \EuF(X,-\omega)$ for the exact Fukaya category $\EuF$. 
We take this as evidence that this definition of the opposite category is most relevant for Fukaya categories. 
We thank Seidel for drawing our attention to this difference.
\end{rmk}

\begin{rmk}\label{rmk:minusop}
Definitions \ref{defn:ainfdg}, \ref{defn:oppdg} and \ref{defn:oppainf} are compatible, but in a slightly non-trivial way: given a $\mathsf{dg}$ category $\EuC$, there is a strict isomorphism of $A_\infty$ categories $A_\infty(\EuC^{op})  \cong  \left(A_\infty(\EuC)\right)^{op}$ which sends $x \mapsto -x$ for all morphisms $x$. 
\end{rmk}

\subsection{Hochschild cohomology}

We define the \emph{Gerstenhaber bracket} on $CC^\bullet(\EuC)$:
\[ [\varphi,\psi] := \varphi \circ \psi - (-1)^{|\varphi|' \cdot |\psi|'} \psi \circ \varphi.\]
It is a graded Lie bracket. 
We define the \emph{Hochschild differential} $M^1: CC^\bullet(\EuC) \to CC^\bullet(\EuC)[1]$ by $M^1 := [\mu^*,-]$. 
Because $[\mu^*,\mu^*] = 0$, $M^1$ is a differential, i.e., $(M^1)^2 = 0$. 
Its cohomology is called the Hochschild cohomology, $\HH^\bullet(\EuC)$.

For $p \ge 2$, we define $M^p \in CC^2(CC^\bullet(\EuC),CC^\bullet(\EuC))^p $ by
\begin{align}
M^p(\varphi_1,\ldots,\varphi_p)(a_1,\ldots,a_s) &:= 
 \sum (-1)^\dagger \mu^*(a_1,\ldots,\varphi_1^*(a_{j_1+1}, \ldots), \ldots,\varphi_p^*(a_{j_p+1},\ldots),\ldots,a_s), \\
\nonumber\text{ where }\quad \dagger &= \sum_{i=1}^p |\varphi_i|' \cdot \epsilon_{j_i}.
 \end{align}
By \cite[Proposition 1.7]{Getzler1993}, the operations $M^* \in CC^2(CC^\bullet(\EuC))$ define an $A_\infty$ structure on $CC^\bullet(\EuC)$. 
In particular, the \emph{Yoneda product} on $\HH^\bullet(\EuC)$, defined on the cochain level by
\begin{equation}
\label{eqn:Yonprod} \varphi \cup \psi := (-1)^{| \psi|} M^2(\psi,\varphi),
 \end{equation}
makes $\HH^\bullet(\EuC)$ into a graded associative algebra. 
Together with the Gerstenhaber bracket, this makes Hochschild cohomology into a \emph{Gerstenhaber algebra}.

\begin{rmk}
We have an isomorphism of $A_\infty$ algebras $CC^\bullet(\EuC) \to CC^\bullet(\EuC^{op})^{op}$, sending $\eta  \mapsto \eta_{op}$.
\end{rmk}

We now define the \emph{Kodaira--Spencer map}, which is closely related to the Kaledin class \cite{Kaledin2007,Lunts2010b}. 
We make a choice of $\BbK$-basis for each morphism space $hom^\bullet_\EuC(X,Y)$. 
We write each $A_\infty$ structure map $\mu^*$ in this basis, as a matrix with entries in $\BbK$. 
We obtain a Hochschild cochain $v(\mu^*) \in CC^2(\EuC)$ by acting with the derivation $v$ on the entries of the matrix for $\mu^*$. 
This Hochschild cochain is closed (as one sees by applying $v$ to the $A_\infty$ equations for $\EuC$), so we may define
\begin{align*}
\mathsf{KS}:\deriv_\Bbbk \BbK &\to \HH^2(\EuC),\\
\mathsf{KS}(v) &:= [v(\mu^*)].
\end{align*}
We will prove in Corollary \ref{cor:ksbasesindep} that the Kodaira--Spencer map is independent of the choice of $\BbK$-bases for the morphisms spaces.

\subsection{Hochschild homology}
\label{subsec:hochhom}

We define the \emph{Hochschild chain complex} 
\[CC_\bullet(\EuC) := \bigoplus_{X_0,\ldots,X_s} \EuC(X_0,\ldots,X_s,X_0)[-1].\]
We denote generators by $a_0[a_1|\ldots|a_s] := a_0 \otimes \ldots \otimes a_s$.
For such a generator, we define the sign
\[ \varepsilon_j := |a_0|' + |a_1|' + \ldots + |a_j|'\]
(the only difference between $\epsilon_j$ and $\varepsilon_j$ is that the former starts at $1$, the latter starts at $0$).

We define an operation $t$ on $CC_\bullet(\EuC)$ by
\[ t(a_0[a_1|\ldots|a_s]) := (-1)^{|a_s|' \cdot \varepsilon_{s-1}} a_s[a_0|\ldots|a_{s-1}].\]

\begin{nota}
\label{not:cycsum}
If $P: CC_\bullet(\EuC) \to M$ is some map, we define
\begin{align*} 
P(a_0[a_1|\ldots|\overbrace{a_{j+1}|\ldots|a_k}|\ldots|a_s]) &:= \sum_{i=j+1}^k P \circ t^i \\
&= \sum_{i=j+1}^k (-1)^\dagger P(a_{s-i+1}[\ldots|a_s|a_0|\ldots|a_{s-i}]),\quad\text{ where}\\
\dagger &:= (\varepsilon_s - \varepsilon_{s-i}) \cdot \varepsilon_{s-i}.
\end{align*}
In words, we add up all ways of cyclically permuting the inputs of $P$, in such a way that $a_0$ lands underneath the brace. We have included Examples \ref{eg:cycsum1} and \ref{eg:cycsum2} to help familiarize the reader with the notation.
\end{nota}

We define the Hochschild differential $b: CC_\bullet(\EuC) \to  CC_\bullet(\EuC)[1]$ by
\begin{multline}
\label{eqn:ccdiffb}
b(a_0[a_1|\ldots|a_s]) :=  \sum_j \mu^*(\overbrace{a_0,\ldots,a_j})[a_{j+1}|\ldots|a_s] 
+ \sum_{j,k} (-1)^{\varepsilon_j}a_0[\ldots|\mu^*(a_{j+1},\ldots,a_k)|\ldots|a_s].
\end{multline}
It is a differential, and its cohomology is called the Hochschild homology, $\HH_\bullet(\EuC)$.

\begin{rmk}
\label{rmk:tautid}
The convention for Hochschild homology of a $\mathsf{dg}$ category $\EuC$ (see, e.g., \cite[Equation (2.1)]{Shklyarov2012}) coincides with ours, i.e., there is a natural identification of cochain complexes
\[(C_\bullet(\EuC),b) = (CC_\bullet(A_\infty(\EuC)),b).\]
\end{rmk}

\begin{defn}
\label{defn:ccopp}
There is an isomorphism of cochain complexes
\begin{align*}
 CC_\bullet(\EuC) & \to CC_\bullet(\EuC^{op})\\
 \alpha &\mapsto \alpha^\vee,\quad\text{sending}\\
 a_0[a_1|\ldots|a_s] &\mapsto (-1)^\dagger a_0[a_s|\ldots|a_1],\quad\text{ where}\\
 \dagger& := \sum_{1 \le i < j \le s} |a_i|' \cdot |a_j|'.
 \end{align*}
\end{defn}

\begin{rmk}
\label{rmk:veesign}
Definition \ref{defn:ccopp} is compatible with the corresponding map for $\mathsf{dg}$ categories, in the sense that for any $\mathsf{dg}$ category $\EuC$, the following diagram commutes up to an overall sign $-1$:
\[ \xymatrix{ C_\bullet(\EuC) \ar[r] \ar[d] & CC_\bullet(A_\infty(\EuC)) \ar[d] \\
C_\bullet(\EuC^{op}) \ar[r] & CC_\bullet(A_\infty(\EuC^{op})).}\]
Here, the horizontal maps are the tautological identifications (see Remark \ref{rmk:tautid}), the left vertical arrow is the map defined, e.g., in \cite[Proposition 4.5]{Shklyarov2012}, and the right vertical arrow is the map defined in Definition \ref{defn:ccopp}, composed with the isomorphism $CC_\bullet(A_\infty(\EuC)^{op}) \cong CC_\bullet(A_\infty(\EuC^{op}))$ induced by the isomorphism of Remark \ref{rmk:minusop}.
\end{rmk}

For $p \ge 1$, we define the operations\footnote{
Getzler denotes $b^{p|1}$ by $b \{-,\ldots,-\}$.}
\[ b^{p|1}: CC^\bullet(\EuC) ^{\otimes p} \otimes CC_\bullet(\EuC) \to CC_\bullet(\EuC)[1-p]\]
\begin{multline}
\label{eqn:bp1}
 b^{p|1}(\varphi_1,\ldots,\varphi_p|a_0,\ldots,a_s) := \\
 \sum (-1)^\dagger \mu^*(a_0,\ldots,\varphi_1^*(a_{j_1+1},\ldots),\ldots,\varphi_p^*(a_{j_p+1},\ldots),\overbrace{\ldots,a_k})[\ldots|a_s], \quad\text{ where}
\end{multline}
\[\dagger = \sum_{i=1}^p |\varphi_i|' \cdot \varepsilon_{j_i}.\]

\begin{example}\label{eg:cycsum1}
We write out an example:
\begin{multline}
b^{1|1}(\varphi|a_0,a_1) = \mu^3(\varphi^0,a_0,a_1) + (-1)^{|a_0|'\cdot |a_1|'} \mu^3(\varphi^0,a_1,a_0) \\
+ (-1)^{(|\varphi|'+|a_0|') \cdot |a_1|'} \mu^3(a_1,\varphi^0,a_0) + (-1)^{|a_0|'\cdot |a_1|'} \mu^2(\varphi^1(a_1),a_0) + \mu^2(\varphi^0,a_0)[a_1].
\end{multline}
\end{example}
By \cite[Theorem 1.9]{Getzler1993}, the operations $b^{p|1}$ (with $b^{0|1} := b$) equip $CC_\bullet(\EuC)$ with the structure of an $A_\infty$ left-module over the $A_\infty$ algebra $CC^\bullet(\EuC)$. 
In particular, $\HH_\bullet(\EuC)$ is a graded left $\HH^\bullet(\EuC)$-module, with the module structure given on the level of cohomology by $ \varphi \cap \alpha := (-1)^{|\varphi|} b^{1|1}(\varphi|\alpha)$.

\begin{lem}
\label{lem:FHH}
If $F: \EuC \to \EuD$ is an $A_\infty$ functor, there is a chain map
\begin{align*} 
F_*: CC_\bullet(\EuC) &\to CC_\bullet(\EuD) \\
F_*(a_0[\ldots|a_s]) &:= \sum F^*(\overbrace{a_0,\ldots})[F^*(\ldots)|\ldots|F^*(\ldots,a_s)].
\end{align*}
\end{lem}

\subsection{Cyclic homology}

If $\EuC$ is an $A_\infty$ category, we define a new $A_\infty$ category $\EuC^+$, by 
\[ hom^\bullet_{\EuC^+}(X,Y) := \left\{ \begin{array}{rl}
						hom^\bullet_{\EuC}(X,Y) & \mbox{ if $X \neq Y$} \\
						hom^\bullet_{\EuC}(X,X) \oplus \BbK \cdot e^+ & \mbox{ if $X=Y$.}
						\end{array} \right. \]
We define $\mu^s(\ldots,e^+,\ldots) = 0$ for all $s \neq 2$, and $\mu^2(e^+,a) = a = (-1)^{|a|}\mu^2(a,e^+) $, 
leaving all other structure maps $\mu^*$ unchanged.
Then $\EuC^+$ is a strictly unital $A_\infty$ category, with strict units $e^+$. 

\begin{rmk}
\label{rmk:oppunit}
There is a strict isomorphism $(\EuC^{op})^+ \cong (\EuC^+)^{op}$, which sends $a \mapsto a$ for all $a \in \hom^\bullet_\EuC$, but sends $e^+ \mapsto -e^+$.
\end{rmk}

If $F:\EuC \to \EuD$ is an $A_\infty$ functor, then we can extend $F$ to an $A_\infty$ functor $F^+: \EuC^+ \to \EuD^+$ by setting $F^1(e^+) = e^+$ and $F^s(\ldots,e^+,\ldots) = 0$ for all $s \ge 2$. 
We define the subcomplex $D_\bullet \subset CC_\bullet(\EuC^+)$ of \emph{degenerate elements}, generated by $a_0[\ldots|a_s]$ such that $a_i = e^+$ for some $i>0$, together with the length-zero chains $e^+$. 
We define the \emph{non-unital Hochschild chain complex}, $CC_\bullet^{nu}(\EuC) := CC_\bullet(\EuC^+)/D_\bullet$. 
When $\EuC$ is cohomologically unital, the composition of the natural maps
\[ CC_\bullet(\EuC) \hookrightarrow CC_\bullet(\EuC^+) \to CC_\bullet^{nu}(\EuC)\]
is a quasi-isomorphism (compare \cite[Section 1.4]{Loday1998}).

Now define Connes' differential $B: CC_\bullet^{nu}(\EuC) \to CC_\bullet^{nu}(\EuC)$ by
\begin{equation}
\label{eqn:ConnesB}
 B(a_0[\ldots|a_s]) := e^+[\overbrace{a_0|\ldots|a_s}].
\end{equation}
It has degree $-1$, and satisfies $B^2 = 0$ and $bB+Bb = 0$.
Therefore, for any graded $\BbK[u]$-module $W$, where $u$ has degree $+2$, we obtain a graded cochain complex $(CC_\bullet^{nu}(\EuC) \otimes W,b+uB)$. 

\begin{rmk}
\label{rmk:dgainfcyc}
The tautological identification of Hochschild complexes for a $\mathsf{dg}$ category, and for its $A_\infty$ version (Remark \ref{rmk:tautid}) equates the $\mathsf{dg}$ version of Connes' differential (with conventions as in \cite[Section 2.2]{Shklyarov2013}) with the $A_\infty$ version \eqref{eqn:ConnesB}.
\end{rmk}

\begin{defn}
We recall the automorphism of $\BbK[u]$ which sends $f \mapsto f^\star$, where $f^\star(u) := f(-u)$.
If $W_1$ and $W_2$ are $\BbK[u]$-modules, we call a map $g: W_1 \to W_2$ \emph{sesquilinear} if
\[ g(f \cdot w) = f^\star \cdot g(w).\]
\end{defn}

\begin{rmk}
\label{rmk:Getzsigns}
Given a sesquilinear automorphism of $W$, also denoted $w \mapsto w^\star$, we obtain an isomorphism of cochain complexes
\[ (CC_\bullet^{nu}(\EuC) \otimes W,b+uB) \cong (CC_\bullet^{nu}(\EuC) \otimes W,b - uB)\]
by sending $\alpha \otimes w \mapsto \alpha \otimes w^\star$. 
Thus, although Getzler uses the convention that the cyclic differential is $b-uB$ in \cite{Getzler1993}, every formula he writes can be translated into our conventions by setting $u \mapsto -u$. 
\end{rmk}

\begin{rmk}
The isomorphism of Definition \ref{defn:ccopp} extends to an isomorphism $CC_\bullet^{nu}(\EuC) \cong CC_\bullet^{nu}(\EuC^{op})$ which intertwines $B$ with $-B$. 
This is a consequence of Remark \ref{rmk:oppunit}: insertion of $e^+$ in $CC_\bullet^{nu}(\EuC)$ corresponds to insertion of $-e^+$ in $CC_\bullet(\EuC^{op})$. 
\end{rmk}

\begin{rmk}
\label{rmk:sesquiop}
As a consequence of the previous two remarks, for any $\BbK[u]$-module $W$ equipped with a sesquilinear automorphism, we obtain a sesquilinear isomorphism of cochain complexes
\begin{eqnarray}
(CC_\bullet^{nu}(\EuC) \otimes W, b+uB) & \to & (CC_\bullet^{nu}(\EuC^{op}) \otimes W,b+uB)\\
\nonumber\alpha \otimes w & \mapsto & \alpha^\vee \otimes w^\star.
\end{eqnarray}
\end{rmk}

If a graded $\BbK[u]$-module $W$ admits, furthermore, an exhaustive decreasing filtration $\ldots \supset F^{\ge p}W \supset F^{\ge p+1} W \supset \ldots$, such that multiplication by $u$ increases the filtration: $u \cdot F^{\ge p} \subset F^{\ge p+1}$, then the cochain complex $(CC_\bullet^{nu}(\EuC) \otimes W,b+uB)$ admits an exhaustive decreasing filtration $CC_\bullet^{nu}(\EuC) \otimes F^{\ge p}W$; so we can take the completion of this filtration in the category of graded cochain complexes, to obtain a new filtered cochain complex $(CC_\bullet^{nu}(\EuC) \widehat{\otimes} W,b+uB)$. 
The cohomology of this cochain complex will also acquire a filtration, which we call the \emph{Hodge filtration} and denote by $F^{\ge p}$. 
The corresponding spectral sequence has $E_1$ page
\begin{equation}
\label{eqn:HdR1}
 E_1^{pq} \cong \bigoplus_r \HH_{p+q-r}(\EuC) \otimes \mathsf{Gr}^pW_r.
\end{equation}

\begin{lem}\label{lem:sscomp}
If $G: CC_\bullet^{nu}(\EuC) \to CC_\bullet^{nu}(\EuD)$ is a map such that $G \circ b = b \circ G$ and $ G \circ B = B \circ G$, then we obtain a map of filtered cochain complexes:
\[ G \widehat{\otimes} W: CC_\bullet^{nu}(\EuC) \widehat{\otimes} W \to CC_\bullet^{nu}(\EuD) \widehat{\otimes} W.\]
If $G$ is a quasi-isomorphism, then $G \widehat{\otimes} W$ is a quasi-isomorphism.
\end{lem}
\begin{proof}
The existence of $G \widehat{\otimes} W$ is clear. 
Because $G \widehat{\otimes} W$ respects filtrations, it induces a map between the corresponding spectral sequences \eqref{eqn:HdR}; because $G$ is a quasi-isomorphism, the map is an isomorphism on the $E_1$ page. 
Therefore, because the filtrations are exhaustive and complete, $G \widehat{\otimes} W$ is a quasi-isomorphism by the Eilenberg--Moore comparison theorem \cite[Theorem 5.5.11]{Weibel1994}. 
\end{proof}

\begin{defn}
The following examples are of particular interest:
\begin{itemize}
\item $W^-:= \BbK[u]$, with filtration $F^{\ge p}W^- := u^p\BbK[u]$. We denote
\[ CC^{-}_\bullet(\EuC) := CC_\bullet^{nu}(\EuC) \widehat{\otimes} W^-,\]
and its cohomology by $\HC^-_\bullet(\EuC)$. This is called the \emph{negative cyclic homology}.
\item $W^\infty := \BbK[u,u^{-1}]$, with the same filtration. We denote
\[ CC^\infty_\bullet(\EuC) := CC_\bullet^{nu}(\EuC) \widehat{\otimes} W^\infty,\]
and its cohomology by $\HP_\bullet(\EuC)$. This is called the \emph{periodic cyclic homology}.
\item $W^+ := \BbK[u,u^{-1}]/ \BbK[u]$, with the same filtration. We denote
\[ CC^{+}_\bullet(\EuC) := CC_\bullet^{nu}(\EuC) \widehat{\otimes} W^+,\]
and its cohomology by $\HC^+_\bullet(\EuC)$. This is called the \emph{positive cyclic homology}.
\end{itemize}
\end{defn}

\begin{rmk}
For these examples, the spectral sequence \eqref{eqn:HdR1} has $E_1$ page
\begin{equation}
\label{eqn:HdR}
 E_1^{pq} \cong \left\{ \begin{array}{rl}
				\HH_{q-p}(\EuC) \cdot u^p & \mbox{ if $u^p \in W$} \\
				0 & \mbox{ otherwise.}
				\end{array} \right.
\end{equation}
\end{rmk}

\begin{rmk}
\label{rmk:cycopp}
As a consequence of Remark \ref{rmk:sesquiop}, there is a sesquilinear isomorphism 
\begin{align*}
\HC^-_\bullet(\EuC) & \to  \HC^-_\bullet(\EuC^{op})\\
\alpha & \mapsto  \alpha^\vee, 
\end{align*}
defined by $(\alpha \otimes w)^\vee := \alpha^\vee \otimes w^\star$, and similarly for $\HP_\bullet$ and $\HC^+_\bullet$.
\end{rmk}

\begin{lem}
\label{lem:FHC}
If $F: \EuC \to \EuD$ is an $A_\infty$ functor, then the map $F^+_*: CC_\bullet^{nu}(\EuC) \to CC_\bullet^{nu}(\EuD)$ satisfies 
\[ F_*^+ \circ b = b \circ F_*^+ \quad\text{and}\quad F_*^+ \circ B = B \circ F_*^+.\]
In particular, it induces a map $F_*: \HC^-_\bullet(\EuC) \to \HC^-_\bullet(\EuD)$, and similarly for $\HP_\bullet$ and $\HC_\bullet^+$.
\end{lem}

As a consequence of Lemma \ref{lem:sscomp}, we have:

\begin{cor}
\label{cor:cycmorita}
If an $A_\infty$ functor $F: \EuC \to \EuD$ induces an isomorphism $F_*: \HH_\bullet(\EuC) \to \HH_\bullet(\EuD)$, then it also induces an isomorphism $F_*: \HC^-_\bullet(\EuC) \to \HC_\bullet^-(\EuD)$, and similarly for $\HP_\bullet$ and $\HC_\bullet^+$. 
\end{cor}

\subsection{The Getzler--Gauss--Manin connection}
\label{subsec:GGM}

Getzler \cite{Getzler1993} defines operations\footnote{Getzler denotes $B^{p|1}$ by $B\{-,\ldots,-\}$.}
\[ B^{p|1}:CC^\bullet(\EuC)^{\otimes p} \otimes CC_\bullet^{nu}(\EuC) \to CC_\bullet^{nu}(\EuC)\]
\begin{align*}
 B^{p|1}(\varphi_1,\ldots,\varphi_p|a_0,\ldots,a_s) &:= \sum (-1)^\dagger e^+[a_0|\ldots|\varphi_1^*(a_{j_1+1},\ldots)|\ldots|\varphi_p^*(a_{j_p+1},\ldots)|\overbrace{\ldots|a_s} ],\quad\text{ where}\\
\dagger &:= \sum_{i=1}^p  |\varphi_i|' \cdot \varepsilon_{j_i}.
\end{align*}

\begin{example}\label{eg:cycsum2}
Note that $B^{0|1} = B$. We also write out another example:
\begin{multline} B^{1|1}(\varphi|a_0,a_1) = e^+[\varphi^0|a_0|a_1] +(-1)^{|a_0|' \cdot|a_1|'} e^+[\varphi^0|a_1|a_0] \\+(-1)^{(|a_0|'+|\varphi|')\cdot|a_1|'}e^+[a_1|\varphi^0|a_0] + (-1)^{|a_0|' \cdot|a_1|'} e^+[\varphi^1(a_1)|a_0] .\end{multline}
\end{example}

\begin{defn}
\label{defn:conn}
The Getzler--Gauss--Manin connection \cite[Proposition 3.1]{Getzler1993} is defined by 
\begin{eqnarray*}
\nabla: \deriv_\Bbbk \BbK \otimes_{\BbK} CC_\bullet^{-}(\EuC) & \to & u^{-1} CC_\bullet^{-}(\EuC),\\
\nabla_v(\alpha) &:=& v(\alpha) - u^{-1} b^{1|1}(v(\mu^*)|\alpha) - B^{1|1}(v(\mu^*)|\alpha).
\end{eqnarray*} 
Observe that the second term on the right-hand side has acquired a minus sign in our conventions, in accordance with Remark \ref{rmk:Getzsigns}.
\end{defn}

\begin{rmk}
In writing the expressions `$v(\alpha)$' and `$v(\mu^*)$', it is implicit that we have chosen a $\BbK$-basis for each morphism space $hom^\bullet_\EuC(X,Y)$. 
So really we should write `$\nabla^{\mathcal{B}}$', where $\mathcal{B}$ denotes the choice of these bases; however we will prove (Corollary \ref{cor:GMbasis}) that $\nabla^\mathcal{B}$ is independent of the choice of $\mathcal{B}$ on the level of cohomology, so $\mathcal{B}$ can be removed from the notation.
\end{rmk}

\begin{rmk}
Observe that $\nabla^\mathcal{B}_v$ induces a linear map $\mathsf{Gr}_F^p CC_\bullet^-(\EuC) \to \mathsf{Gr}_F^{p-1} CC_\bullet^-(\EuC)$. 
This map is given by $-u^{-1}b^{1|1}(\mathsf{KS}(v)|-)$ on the level of cohomology, in analogy with the associated graded of the Gauss--Manin connection with respect to the Hodge filtration (see, e.g., \cite[Theorem 10.4]{Voisin2002a}).
\end{rmk}

Getzler shows that $[\nabla_v^\mathcal{B},b+uB] = 0$, so $\nabla_v^\mathcal{B}$ gives a well-defined map on the level of cohomology. 
It is clear from the formula that it is a connection. 
Getzler also shows that $\nabla^\mathcal{B}$ is flat: more precisely, he writes down an explicit contracting homotopy for $u^2\left([\nabla_X,\nabla_Y] - \nabla_{[X,Y]}\right)$ 
(see \cite[Theorem 3.3]{Getzler1993}), so the connection is flat in the sense of Definition \ref{defn:prevhsnopair}. 

\begin{thm}
\label{thm:GMMorita}
Suppose that $\EuC$ and $\EuD$ are $A_\infty$ categories, equipped with a choice of $\BbK$-bases $\mathcal{B}_\EuC$ for the morphism spaces of $\EuC$, and $\mathcal{B}_\EuD$ for the morphism spaces of $\EuD$.
If $F: \EuC \to \EuD$ is an $A_\infty$ functor, then the induced map $F_*: \HC^-_\bullet(\EuC) \to \HC^-_\bullet(\EuD)$ respects the Getzler--Gauss--Manin connection, in the sense that
\[ F_* \circ \left[\nabla^{\mathcal{B}_\EuC}_v\right] = \left[\nabla^{\mathcal{B}_\EuD}_v\right] \circ F_*\]
on the level of cohomology.
\end{thm}
\begin{proof}
See Appendix \ref{app:GM}.
\end{proof}

\begin{cor}
\label{cor:GMbasis}
The Getzler--Gauss--Manin connection $\nabla^\mathcal{B}$ is independent of the choice of bases $\mathcal{B}$, on the level of cohomology.
\end{cor}
\begin{proof}
Follows from Theorem \ref{thm:GMMorita}, taking $F$ to be the identity functor.
\end{proof}

Henceforth, we simply write `$\nabla$' instead of `$\nabla^\mathcal{B}$'. 
It follows that, for any graded $\BbK$-linear $A_\infty$ category $\EuC$,  $(\HC_\bullet^-(\EuC),\nabla)$ is a well-defined unpolarized pre-\vshs{}. 
This completes the proof of Theorem \ref{thm:main} \eqref{it:unpprevshs}.

\section{Morita invariance}
\label{sec:mor}

\subsection{Morita equivalence}

We recall some material about $A_\infty$ bimodules from \cite[Section 2]{Seidel2008c}. 
If $\EuC$ and $\EuD$ are $A_\infty$ categories, we denote by $[\EuC,\EuD]$ the $\mathsf{dg}$ category of graded, $\BbK$-linear, cohomologically unital $A_\infty$ $(\EuC,\EuD)$ bimodules. 
Recall: morphisms are `pre-homomorphisms' of bimodules; the differential is given by \cite[Equation (2.8)]{Seidel2008c}; composition is given by \cite[Equation (2.9)]{Seidel2008c}.
 
Recall that if $\EuB$, $\EuC$ and $\EuD$ are $A_\infty$ categories, and $\EuM$ is an $A_\infty$ $(\EuC,\EuD)$ bimodule, then there is an induced $\mathsf{dg}$ functor
\begin{equation}
\label{eqn:moritafun} ? \otimes_\EuC \EuM: [\EuB,\EuC] \to [\EuB,\EuD].
\end{equation}
If $\EuD = \EuC$ and $\EuM = \EuC_\Delta$ is the diagonal bimodule, then the functor $? \otimes_\EuC \EuC_\Delta$ is quasi-isomorphic to the identity functor.

\begin{defn}
$\EuC$ and $\EuD$ are \emph{Morita equivalent} if there exists a $(\EuC,\EuD)$ bimodule $\EuM$, and a $(\EuD,\EuC)$ bimodule $\EuN$, and quasi-isomorphisms of $A_\infty$ bimodules
\[ \EuM \otimes_\EuD \EuN  \cong \EuC_\Delta \qquad \text{and} \qquad \EuN \otimes_\EuC \EuM \cong \EuD_\Delta.\]
In this situation, the functor \eqref{eqn:moritafun} is a quasi-equivalence. 
\end{defn}

We now recall that, given $A_\infty$ functors $F_i: \EuC_i \to \EuD_i$ for $i=0,1$, and a $(\EuD_0,\EuD_1)$ bimodule $\EuM$, we can define the pullback $(\EuC_0,\EuC_1)$ bimodule $(F_0 \otimes F_1)^* \EuM$ (see \cite[Section 2.8]{Ganatra2013}). 
We prove the following result in Appendix \ref{app:morita}:

\begin{lem}[= Lemma \ref{lem:moritafull}]
\label{lem:moritasplit}
If $F: \EuC \to \EuD$ is a cohomologically full and faithful $A_\infty$ functor, and $\EuD$ is split-generated by the image of $F$, then $\EuM := (F \otimes \mathrm{Id})^* \EuD_\Delta$ and $\EuN := (\mathrm{Id} \otimes F)^* \EuD_\Delta$ define a Morita equivalence between $\EuC$ and $\EuD$.
\end{lem}

Now, let $\twsplit \EuC$ denote the triangulated split-closure of $\EuC$ (denoted  `$\prod(Tw(\EuC))$' in \cite[Section 4c]{Seidel2008}). 
The following result is well-known:

\begin{thm}
\label{thm:moritasplit}
$\EuC$ and $\EuD$ are Morita equivalent if and only if $\twsplit \EuC$ and $\twsplit \EuD$ are quasi-equivalent.
\end{thm}
\begin{proof}
Suppose $\twsplit \EuC \simeq \twsplit \EuD$.
Consider the $A_\infty$ functors
\[ \EuC \hookrightarrow \twsplit \EuC \to \twsplit \EuD \hookleftarrow \EuD.\]
Each is cohomologically full and faithful with split-generating image, hence each defines a Morita equivalence by Lemma \ref{lem:moritasplit}.
This proves the `if'; Theorem \ref{thm:moritaonlyif} proves the `only if'.
\end{proof}

\subsection{Hochschild cohomology}

Generalizing \cite[Equation (2.200)]{Ganatra2013} slightly, we have the following:

\begin{lem}
\label{lem:ccbimod}
There are $A_\infty$ homomorphisms\footnote{We recall that $[\EuC,\EuD]$ is a $\mathsf{dg}$  category, and $A_\infty([\EuC,\EuD])$ is the corresponding $A_\infty$ category, in accordance with Definition \ref{defn:ainfdg}.}
\begin{equation}
\label{eqn:LCC} CC^\bullet(\EuC)  \xrightarrow{L_\EuM}  hom^\bullet_{A_\infty([\EuC,\EuD])}(\EuM,\EuM) \xleftarrow{R_\EuM} CC^\bullet(\EuD)^{op},
\end{equation}
with $L_\EuM$ given by the formula
\begin{multline}
 L^p_\EuM(\varphi_1,\ldots,\varphi_p)(a_1,\ldots,a_s,m,b_1,\ldots,b_t) := \\
 \sum (-1)^\dagger \mu^*_\EuM(a_1,\ldots,\varphi_1^*(a_{j_1+1}, \ldots), \ldots,\varphi_p^*(a_{j_p+1},\ldots),\ldots,a_1,m,b_1,\ldots, b_t), \quad\text{ where}
\end{multline}
\[ \dagger = \sum_{i=1}^p |\varphi_i|' \cdot \epsilon_{j_i},\]
and $R_\EuM$ given by the formula
\begin{multline}
 R^p_\EuM(\varphi_1,\ldots,\varphi_p)(a_1,\ldots,a_s,m,b_1,\ldots,b_t) := \\
 \sum (-1)^\dagger \mu^*_\EuM(a_1,\ldots,a_s,m,b_1,\ldots,\varphi_p^*(y_{j_p+1}, \ldots),\ldots,\varphi_1^*(y_{j_1+1},\ldots),\ldots,y_t),\quad\text{ where}
\end{multline}
\[ \dagger = \sum_{i<j} |\varphi_i|' \cdot |\varphi_j|' + \sum_{i=1}^p |\varphi_i|' \cdot \left( |a_1|'+ \ldots+|a_s|'+|m|+|b_1|' + \ldots + |b_{j_i}|'\right).\]
\end{lem}
\begin{proof}
The $A_\infty$ homomorphism equations are a consequence of the $A_\infty$ bimodule equations for $\EuM$.
\end{proof}

\begin{lem}
If $\EuM$ defines a Morita equivalence between $\EuC$ and $\EuD$, then $L_\EuM$ and $R_\EuM$ are quasi-isomorphisms. 
In particular, $\EuM$ induces an algebra isomorphism $\HH^\bullet(\EuC) \cong \HH^\bullet(\EuD)^{op}$.
\end{lem}
\begin{proof}
It suffices to prove that the chain maps $L^1_\EuM$ and $R^1_\EuM$ are quasi-isomorphisms. 
We start by observing that the following diagram of chain maps commutes up to homotopy:
\begin{equation}
\label{eqn:LRcomm}
\xymatrix{ CC^\bullet(\EuC) \ar[r]^-{L^1_\EuM} \ar[d]^{R^1_{\EuC_\Delta}} & hom^\bullet_{[\EuC,\EuD]} (\EuM,\EuM) \ar[d]^{\EuC_\Delta \otimes ?} \\
hom^\bullet_{[\EuC,\EuC]}(\EuC_\Delta,\EuC_\Delta) \ar[r]^-{ ? \otimes_\EuC \EuM} & hom^\bullet_{[\EuC,\EuD]}(\EuC_\Delta \otimes_\EuC \EuM,\EuC_\Delta \otimes_\EuC \EuM).}
\end{equation}
Indeed, the homotopy $H: CC^\bullet(\EuC) \to hom^\bullet_{[\EuC,\EuD]}(\EuC_\Delta \otimes_\EuC \EuM,\EuC_\Delta \otimes_\EuC \EuM)$ is given by
\begin{align*} H(\varphi^*)^{0|1|0}(a_0 , a_1 ,\ldots , a_s , m) &:= \sum (-1)^\dagger (a_0, a_1 , \ldots, \varphi^*(a_{j+1},\ldots),   \ldots , a_s, m),\quad\text{ where}\\
 \dagger &:= |\varphi|' \cdot \varepsilon_j,\\
 H(\varphi^*)^{j|1|k} &=0 \quad\text{for $j>0$ or $k>0$.}
 \end{align*}

We now observe that $R^1_{\EuC_\Delta}$ is a quasi-isomorphism by \cite[Proposition 2.5]{Ganatra2013}. 
$\EuC_\Delta \otimes ?$ is a quasi-isomorphism because it is quasi-isomorphic to the identity functor. 
$? \otimes_\EuC \EuM$ is a quasi-isomorphism because $\EuM$ defines a Morita equivalence. 
Therefore, the chain map $L^1_\EuM$ is a quasi-isomorphism, by commutativity of the diagram. 
The proof that $R^1_\EuM$ is a quasi-isomorphism is analogous. 
\end{proof}

\begin{lem}
The isomorphism $\HH^\bullet(\EuC) \cong \HH^\bullet(\EuC)^{op}$, induced by the diagonal bimodule, is the identity. 
In particular, $\HH^\bullet(\EuC)$ is graded commutative (cf. \cite{Gerstenhaber1963}).
\end{lem}
\begin{proof}
It suffices to check that the chain maps
\[ L^1_{\EuC_\Delta}, -R^1_{\EuC_\Delta}: CC^\bullet(\EuC) \to hom^\bullet_{A_\infty([\EuC,\EuC])}(\EuC_\Delta,\EuC_\Delta)\]
are chain-homotopic (Remark \ref{rmk:minusop} explains the minus sign).
Indeed, the homotopy is given by
\[ H(\varphi)^{s|1|t}(a_1,\ldots,a_s,m,b_1,\ldots,b_t) := \varphi(a_1,\ldots,a_s,m,b_1,\ldots,b_t).\]
\end{proof}

\begin{cor}
\label{cor:hcohmorita}
A Morita equivalence between $\EuC$ and $\EuD$ induces an isomorphism of graded $\BbK$-algebras
\begin{equation}
\label{eqn:hcohmorita}
 \HH^\bullet(\EuC) \cong \HH^\bullet(\EuD).
 \end{equation}
\end{cor}

\begin{prop}
\label{prop:KSmorita}
The isomorphism \eqref{eqn:hcohmorita} respects Kodaira--Spencer maps.
\end{prop}
\begin{proof}
Let $\EuM$ be a $(\EuC,\EuD)$ bimodule which defines a Morita equivalence between $\EuC$ and $\EuD$, and let us choose a basis for the morphisms spaces of $\EuC$, $\EuD$ and $\EuM$. 
Given a derivation $v \in \deriv_\Bbbk \BbK$, we have
\[ 0 = \partial(v(\mu^*_\EuM)) + L^1_\EuM(v(\mu^*_\EuC)) - R^1_\EuM(v(\mu^*_\EuD)),\]
as follows by applying $v$ to the $A_\infty$ bimodule equations for $\EuM$. 
Therefore $L^1_\EuM(v(\mu^*_\EuC)) = R^1_\EuM(v(\mu^*_\EuD))$ on the level of cohomology, and the result follows.
\end{proof}

Taking $\EuC=\EuD$ and $\EuM=\EuC_\Delta$, we obtain:

\begin{cor}\label{cor:ksbasesindep}
The class $\mathsf{KS}(v) := [v(\mu^*)] \in \HH^2(\EuC)$ does not depend on the choice of $\BbK$-bases for the morphism spaces of $\EuC$.
\end{cor}

\subsection{Hochschild homology}

We recall the notion of cyclic tensor product of bimodules, from \cite[Section 5]{Seidel2008c}. 
If $\EuC_1,\ldots,\EuC_l=\EuC_0$ are $A_\infty$ categories, and $\EuM_i$ a $(\EuC_{i-1},\EuC_i)$ bimodule for $i=1\ldots,l$, we can form the cyclic tensor product $\EuM_1 \otimes_{\EuC_1} \EuM_2 \otimes_{\EuC_2} \ldots \otimes_{\EuC_{l-1}}\EuM_l \otimes_{\EuC_l} cyc$. 
It is a chain complex with underlying vector space
\[ \bigoplus_{X_{i,j} \in Ob(\EuC_i)} \EuM_1(X_{0,j_0},X_{1,1}) \otimes \EuC_1(X_{1,1},\ldots,X_{1,j_1}) \otimes \EuM_2(X_{1,j_1},X_{2,1}) \otimes \ldots \otimes \EuM_l(X_{l-1,j_l},X_{0,1}) \otimes \EuC_0(X_{0,1},\ldots,X_{0,j_0}),\]
and differential as in \cite[Equation (5.1)]{Seidel2008c}. 
As a particular case, we have the identification $CC_\bullet(\EuC) = \EuC_\Delta \otimes_\EuC cyc$.

\begin{lem}
\label{lem:hochhommor}
A Morita equivalence between $\EuC$ and $\EuD$ induces an isomorphism of graded vector spaces
\begin{equation}
\label{eqn:hhommorita}
\HH_\bullet(\EuC) \cong \HH_\bullet(\EuD).
\end{equation}
\end{lem}
\begin{proof}
Let $\EuM$ be a $(\EuC,\EuD)$ bimodule and $\EuN$ a $(\EuD,\EuC)$ bimodule which define a Morita equivalence between $\EuC$ and $\EuD$. 
Then we have a chain of quasi-isomorphisms
\[CC_\bullet(\EuC)  = \EuC_\Delta \otimes_\EuC cyc.  \simeq \EuM \otimes_\EuD \EuN \otimes_\EuC cyc.  \simeq \EuN \otimes_\EuC \EuM \otimes_\EuD cyc. \simeq \EuD_\Delta \otimes_\EuD cyc. = CC_\bullet(\EuD).\]
\end{proof}

\begin{rmk}
The isomorphism \eqref{eqn:hhommorita} respects the module structure over Hochschild cohomology.
\end{rmk}

\begin{lem}
\label{lem:hhommormap}
Let $F: \EuC \to \EuD$ is an $A_\infty$ functor, so that the $(\EuC,\EuD)$ bimodule $\EuM := (F \otimes \mathrm{Id})^*\EuD_\Delta$ and the $(\EuD,\EuC)$ bimodule $\EuN := (\mathrm{Id} \otimes F)^* \EuD_\Delta$ define a Morita equivalence between $\EuC$ and $\EuD$ (compare Lemma \ref{lem:moritasplit}). 
Then the induced isomorphism \eqref{eqn:hhommorita} coincides with the map $F_*$ defined in Lemma \ref{lem:FHH}.
\end{lem}
\begin{proof}
The key point is to check that the maps $\EuM \otimes_\EuD \EuN \otimes_\EuC cyc. \to \EuD_\Delta \otimes_\EuD cyc.$ 
given by
\begin{multline}
 m[ b_1|\ldots|b_t]n[a_1|\ldots|a_s] \mapsto\\
 \sum \mu^*(b_{j+1},\ldots,b_t,n,F^*(a_1,\ldots),\ldots,F^*(\ldots,a_s),m,b_1,\ldots)[b_{k+1}|\ldots|b_j]
\end{multline}
and
\begin{multline}
 m[b_1|\ldots|b_t]n[a_1|\ldots|a_s] \mapsto \\
 \sum \mu^*(\ldots,F^*(\ldots,a_s),m,b_1,\ldots,b_t,n,F^*(a_1,\ldots),\ldots)[F^*(a_{k+1},\ldots)|\ldots|F^*(\ldots,a_j)]
\end{multline}
are chain homotopic. 
The chain homotopy is given by
\begin{equation*}
 H(m[ b_1|\ldots|b_t]n[a_1|\ldots|a_s]) := m[b_1|\ldots|b_t|n|F^*(a_1,\ldots)|\ldots|F^*(\ldots,a_s)].
 \end{equation*}
\end{proof}

\begin{cor}
\label{cor:cycGMmor}
$\HC_\bullet^-(\EuC), \HP_\bullet(\EuC)$ and $\HC_\bullet^+(\EuC)$ are Morita invariants. 
So is the Getzler--Gauss--Manin connection.
\end{cor}
\begin{proof}
Suppose $\EuC$ and $\EuD$ are Morita equivalent. 
It follows by Theorem \ref{thm:moritasplit} that we have $A_\infty$ functors
\[ \EuC \hookrightarrow \twsplit \EuC \to \twsplit \EuD \hookleftarrow \EuD.\]
Each of these induces a map on Hochschild and cyclic homology, by Lemmas \ref{lem:FHH} and \ref{lem:FHC}. 
Furthermore, the maps on Hochschild homology coincide with the corresponding maps \eqref{eqn:hhommorita}, by Lemma \ref{lem:hhommormap}; so they are isomorphisms, by Lemma \ref{lem:hochhommor}. 
Therefore, the induced maps on cyclic homology are isomorphisms, by Corollary \ref{cor:cycmorita}: furthermore, they respect the Getzler--Gauss--Manin connections, by Theorem \ref{thm:GMMorita}.
\end{proof}

\section{Pairings on Hochschild and cyclic homology} 
\label{sec:mukai}

\subsection{The Mukai pairing for $\mathsf{dg}$ categories}

Let $\EuC$ be a $\BbK$-linear $\mathsf{dg}$ category. 
We recall a construction due to Shklyarov \cite{Shklyarov2012}. 

There is a natural notion of tensor product of $\BbK$-linear $\mathsf{dg}$ categories, and there is a K\"{u}nneth quasi-isomorphism of Hochschild chain complexes \cite[Theorem 2.8]{Shklyarov2012}
\begin{equation}
\label{eqn:kunndg}
C_\bullet(\EuC) \otimes C_\bullet(\EuD) \to C_\bullet(\EuC \otimes \EuD).
\end{equation}

If $\EuC$ and $\EuD$ are $\mathsf{dg}$ categories, then a $\mathsf{dg}$ $(\EuC,\EuD)$ bimodule $\EuQ$ consists of the following data: for each pair $(X,Y) \in Ob(\EuC) \times Ob(\EuD)$, a graded $\BbK$-vector space $\EuQ^\bullet(X,Y)$ equipped with a differential $d$ of degree $+1$; left-module maps
\begin{eqnarray*}
hom^\bullet_\EuC(X_1,X_2) \otimes \EuQ^\bullet(X_0,X_1) &\to& \EuQ^\bullet(X_0,X_2) \\
f \otimes q & \mapsto & f \cdot q;
\end{eqnarray*}
and right-module maps
\begin{eqnarray*}
\EuQ^\bullet(X_1,X_2) \otimes hom^\bullet_{\EuD}(X_0,X_1) &\to& \EuQ^\bullet(X_0,X_2) \\
q \otimes g & \mapsto & q \cdot g
\end{eqnarray*}
satisfying the obvious analogues of associativity \eqref{eqn:assoc}, the Leibniz rule \eqref{eqn:leibniz} and unitality \eqref{eqn:dgunit}.

A $\mathsf{dg}$ $(\EuC,\EuD)$ bimodule $\EuP$ is equivalent to a $\mathsf{dg}$ functor $\EuP: \EuC \otimes \EuD^{op} \to \modules \BbK$. 
On the level of objects, the functor sends $(X,Y) \mapsto \EuP^\bullet(X,Y)$. 
To define the functor on the level of morphisms, we first define, for any $c \in hom^\bullet_\EuC(X_1,X_2)$,
\begin{eqnarray*}
L(c): \EuP(X_0,X_1) &\to& \EuP(X_0,X_2),\\
p & \mapsto & c \cdot p.
\end{eqnarray*}
Similarly, for any $d \in hom^\bullet_\EuD(X_0,X_1)$, we define
\begin{eqnarray*}
R(d): \EuP(X_1,X_2) &\to& \EuP(X_0,X_2),\\
p & \mapsto & (-1)^{|p| \cdot |d|} p \cdot d.
\end{eqnarray*}
We then define the functor on the level of morphisms: $\EuP(c \otimes d) := L(c) \circ R(d)$.

By functoriality of Hochschild homology, a $\mathsf{dg}$ $(\EuC,\EuD)$ bimodule $\EuP$ induces a chain map
\[C_\bullet(\EuC \otimes \EuD^{op}) \to C_\bullet(\modules \BbK). \]
Pre-composing this with the K\"{u}nneth quasi-isomorphism \eqref{eqn:kunndg} gives another chain map, which induces a map on cohomology
\begin{equation}
\label{eqn:wedgepdg} \wedge_\EuP: \HH_\bullet(\EuC) \otimes \HH_\bullet(\EuD^{op}) \to \HH_\bullet(\modules \BbK).
\end{equation}

Now we consider the full $\mathsf{dg}$ sub-category $\perfdg \BbK \subset \modules \BbK$, whose objects are the cochain complexes with finite-dimensional cohomology. 
There is an obvious $\mathsf{dg}$ functor $\BbK \hookrightarrow \perfdg \BbK$ given by including the full subcategory with the single object $\BbK[0]$. 
This induces an isomorphism 
\begin{equation}
\label{eqn:FLSdef}
 \BbK \cong \HH_\bullet(\BbK) \overset{\cong}{\longrightarrow} \HH_\bullet(\perfdg \BbK),
\end{equation}
whose inverse is called the `Feigin--Losev--Shoikhet trace' in \cite{Shklyarov2012}:
\begin{equation}
\label{eqn:FLS}
 \int: \HH_\bullet(\perfdg \BbK) \overset{\cong}{\longrightarrow} \BbK.
\end{equation}

\begin{defn}
\label{defn:dgmukai}
We call a $\mathsf{dg}$ $(\EuC,\EuD)$ bimodule $\EuP$ \emph{proper} if $\EuP(X,Y) \in \perfdg \BbK$ for all $(X,Y) \in Ob(\EuC) \times Ob(\EuD)$. 
A proper bimodule induces a pairing
\begin{eqnarray*}
C_\bullet(\EuC) \otimes C_\bullet(\EuD^{op}) & \to & \BbK\\
\alpha \otimes \beta & \mapsto & \int \wedge_\EuP(\alpha,\beta).
\end{eqnarray*}
If $\EuC$ is a proper $\mathsf{dg}$ category, we call the pairing 
\begin{eqnarray*}
\langle , \rangle_{Muk}:\HH_\bullet(\EuC) \otimes \HH_\bullet(\EuC) & \to & \BbK,\\
\langle \alpha ,\beta\rangle_{Muk} & := & \int \wedge_{\EuC_\Delta}(\alpha,\beta^\vee)
\end{eqnarray*}
the \emph{Mukai pairing}. 
Shklyarov shows that the Mukai pairing is Morita invariant. 
\end{defn}

\begin{lem}
\label{lem:flsstr}
Let $\fink \BbK \subset \perfdg \BbK$ denote the full subcategory whose objects are the finite-dimensional cochain complexes. 
There is a chain map
\begin{align*} 
\mathsf{Str}: C_\bullet(\fink \BbK) &\to \BbK,\quad\text{sending}\\
a_0 & \mapsto  \mathsf{str}(a_0), \\
a_0[a_1|\ldots|a_n] & \mapsto  0 \mbox{ for $n \ge 1$},
\end{align*}
where `$\mathsf{str}$' on the first line denotes the supertrace.\footnote{
If $V$ is a graded $\BbK$-vector space, and $F \in \mathrm{End}_\BbK(V)$ an endomorphism, we define $\mathsf{str}(F)$ as follows: write $F$ as a sum of components $F_{pq} \in \Hom_\BbK(V_p,V_q)$, then $\mathsf{str}(F) := \sum_p (-1)^p \mathsf{tr}(F_{pp})$.
}
It induces a map $\mathsf{Str}:\HH_\bullet(\fink \BbK) \to \BbK$; this coincides with the composition
\[ \HH_\bullet(\fink \BbK) \to \HH_\bullet(\perfdg \BbK) \overset{\int}{\to} \BbK.\]
\end{lem}
\begin{proof}
One easily verifies that $\mathsf{Str}$ is a chain map.
The inclusion $\fink \BbK \hookrightarrow \perfdg \BbK$ is a quasi-equivalence, so induces an isomorphism of Hochschild homologies. 
It is obvious that $\mathsf{Str}$ is left-inverse to the map induced by the inclusion \eqref{eqn:FLSdef}, and the result follows.
\end{proof}

Shklyarov derives the following formula for $\wedge_\EuP$: if $\alpha = a_0[a_1|\ldots|a_s] \in C_\bullet(\EuC)$ and $\beta = b_0[b_1|\ldots|b_t] \in C_\bullet(\EuD^{op})$, then
\begin{equation}
\label{eqn:Shkform}
\wedge_{\EuP} (\alpha,\beta) = (-1)^{|b_0| \cdot (|a_1|'+ \ldots + |a_s|')}  L(a_0)R(b_0)\mathsf{sh}_{st}[L(a_1)|\ldots|L(a_s)|R(b_1)|\ldots|R(b_t)].
\end{equation}
Here, $\mathsf{sh}_{st}$ denotes the sum of all $(s,t)$-shuffles of the elements in the square brackets, with the associated Koszul signs (where interchanging $L(a_i)$ with $R(b_j)$ introduces a sign $|a_i|'\cdot |b_j|'$).
To clarify: the symbols `$L(a_i)$' and `$R(b_j)$' in \eqref{eqn:Shkform} are regarded as morphisms in the $\mathsf{dg}$ category $\modules \BbK$.

\subsection{$A_\infty$ multifunctors}

The notion of tensor product of $A_\infty$ categories is rather involved \cite{Amorim2016}. 
Nevertheless there is a relatively straightforward notion of $A_\infty$ $n$-functor $\EuC_1 \times \ldots \times \EuC_n \dashrightarrow \EuD$, which forms a substitute for the notion of an $A_\infty$ functor $\EuC_1 \otimes \ldots \otimes \EuC_n \dashrightarrow \EuD$, and suffices for many purposes. We give the definition, following \cite{Lyubashenko2012}.

\begin{defn}
\label{defn:multifun}
Let $\EuC_1,\ldots,\EuC_n$ and $\EuD$ be $A_\infty$ categories. 
An \emph{ $A_\infty$ $n$-functor} $F: \EuC_1 \times \ldots \times \EuC_n \dashrightarrow \EuD$ consists of a map  $F: Ob(\EuC_1) \times \ldots \times Ob(\EuC_n) \to Ob(\EuD)$, together with $\BbK$-linear maps
\[ F^{s_1;\ldots;s_n}: \EuC_1(X^1_1,\ldots,X^1_{s_1}) \otimes \ldots \otimes \EuC_n(X^n_1,\ldots,X^n_{s_n}) \to \EuD(F(X^1_1,\ldots,X^n_1),F(X^1_{s_1},\ldots,X^n_{s_n}))\]
of degree $0$, such that $F^{0;0;\ldots;0} = 0$, and satisfying the \emph{$A_\infty$ $n$-functor relations} (a visual representation of which is given in Figure \ref{fig:ainfmultifun}):
\begin{multline}
\sum_{i,j,k} (-1)^\dagger F^{s_1;\ldots;s_i +1-k;\ldots;s_n}(c^1_1,\ldots,c^1_{s_1};\ldots;c^i_1,\ldots,\mu^k_{\EuC_i}(c^i_{j},\ldots),c^i_{j+k+1},\ldots,c^i_{s_i};c^n_1,\ldots,c^n_{s_n})\\
= \sum_{k,i_{p,q}} (-1)^\maltese \mu^k_\EuD(F(c^1_1,\ldots,c^1_{i_{1,1}};\ldots;c^n_1,\ldots,c^n_{i_{n,1}}),\ldots,F(c^1_{i_{1,k}+1},\ldots,c^n_{i_{n,k}+1})).
\end{multline}
The sign $\dagger$ is the Koszul sign obtained by commuting $\mu^k_{\EuC_i}$ (equipped with degree $1$) to the front of the expression (where each $c^p_q$ is equipped with its reduced degree $|c^p_q|'$). 
We henceforth adopt the convention, in expressions involving $A_\infty$ multifunctors, that $(-1)^\maltese$ is the Koszul sign associated to re-ordering the inputs $c^p_q$ in the expression so that they appear in the order $(c^1_1,\ldots,c^1_{s_1};\ldots;c^n_1,\ldots,c^n_{s_n})$ (still equipping the $c^p_q$ with their reduced degrees). 

If $\EuC_i$ is strictly unital (with units denoted $e$), we say that $F$ is \emph{strictly unital in the $i$th entry} if
\[ F^{s_1;\ldots;s_n}(\ldots;a^i_1,\ldots,e,\ldots,a^i_{s_i};\ldots) = \left\{ \begin{array}{rl}
													e & \mbox{ if $s_j = 0$ for all $j \neq i$, and $s_i = 1$;} \\
													0 & \mbox{ otherwise.}
											\end{array} \right.
\]
\end{defn}

\begin{lem}
An $A_\infty$ $n$-functor $F:\EuC_1 \times \ldots \times \EuC_n \dashrightarrow \EuD$ induces a functor
\[ H^\bullet(\EuC_1) \otimes \ldots \otimes H^\bullet(\EuC_n) \to H^\bullet(\EuD)\]
(the tensor product on the left is defined by considering each $H^\bullet(\EuC_i)$ as a $\mathsf{dg}$ category with trivial differential -- in particular the composition involves the Koszul sign rule). 
The action on objects is obvious, and on morphisms it sends
\[ [a_1] \otimes \ldots \otimes [a_n] \mapsto [F^{1;0;\ldots;0}(a_1)] \cdot \ldots \cdot [F^{0;\ldots;0;1}(a_n)].\]
If $F$ is strictly unital in all entries, then this functor is unital.
\end{lem}
\begin{proof}
The components $F^{0;\ldots;0;s;0;\ldots;0}$ of $F$ define $A_\infty$ functors $\EuC_i \dashrightarrow \EuD$ for each $i$, which induce functors $H^\bullet(\EuC_i) \to H^\bullet(\EuD)$ by taking cohomology. 
It now suffices to check that elements in (distinct) images of these functors supercommute, which is a consequence of the following $A_\infty$ $n$-functor relation, written in the case $n=2$ to avoid notational clutter:
\begin{multline}
F^{1;1}(\mu^1(a);b)+(-1)^{|a|'}F^{1;1}(a;\mu^1(b)) =\\
 \mu^2(F^{1;0}(a;),F^{0;1}(;b)) + (-1)^{|a|'\cdot |b|'} \mu^2(F^{0;1}(;b),F^{1;0}(a;)) + \mu^1(F^{1;1}(a;b)).\end{multline}
The unitality part of the claim is straightforward.
\end{proof}

\begin{example}
\label{eg:dgmulti}
Let $\EuC_1,\ldots,\EuC_n$ be $\mathsf{dg}$ categories, and $\EuC_1 \otimes \ldots \otimes \EuC_n$ their tensor product $\mathsf{dg}$ category.
Then there is an $A_\infty$ $n$-functor
\begin{align*} 
F: A_\infty(\EuC_1) \times \ldots \times A_\infty(\EuC_n) &\dashrightarrow A_\infty(\EuC_1 \otimes \ldots \otimes \EuC_n),\quad\text{with}\\ 
 F^{0;\ldots;1;\ldots;0}(;\ldots;c_i;\ldots;) &:= e_1 \otimes \ldots \otimes e_{i-1} \otimes c_i \otimes e_{i+1} \otimes \ldots \otimes e_n,
 \end{align*}
and all other $F^{*;\ldots;*}$ vanishing.
\end{example}

\begin{defn}
\label{defn:multifuncomp}
Suppose that we have $A_\infty$ multifunctors $F_i: \EuC^i_1 \times \ldots \times \EuC^i_{t_i} \dashrightarrow \EuD_i$ for $i=1,\ldots,m$, and $ G: \EuD_1 \times \ldots \times \EuD_m \dashrightarrow \EuE$.
We define the composition
\[ H := G\circ(F_1,\ldots,F_m): \EuC^1_1 \times \ldots \times \EuC^m_{t_m} \dashrightarrow \EuE.\]
It acts on objects in the obvious way, and on morphisms by analogy with composition of $A_\infty$ functors:
\begin{multline}
H^{s_{1,1},\ldots,s_{1,t_1},\ldots,s_{m,1},\ldots,s_{m,t_m}}(c^{1,1}_1,\ldots,c^{1,1}_{s_{1,1}};\ldots;c^{m,t_m}_1,\ldots,c^{m,t_m}_{s_{m,t_m}}) := \\
\sum (-1)^\maltese  G\left(F_1^*(c^{1,1}_1,\ldots;\ldots;c^{1,t_1}_1,\ldots),\ldots,F_1^*(\ldots;\ldots,c^{1,t_1}_{s_{1,t_1}});\ldots;F_m^*(\ldots),\ldots,F_m^*(\ldots;\ldots,c^{m,t_m}_{s_{m,t_m}})\right)
\end{multline}
The check that the maps $H^*$ satisfy the $A_\infty$ multifunctor equations is straightforward. 
It is also easy to check that composition is `associative' in the obvious sense.
\end{defn}

\begin{lem}
\label{lem:trifun}
Let $\EuC$ and $\EuD$ be $\BbK$-linear $A_\infty$ categories, and $[\EuC,\EuD]$ the $\mathsf{dg}$ category of $A_\infty$ $(\EuC,\EuD)$ bimodules. 
There is an $A_\infty$ tri-functor
\[ F: A_\infty([\EuC,\EuD]) \times \EuC \times \EuD^{op} \dashrightarrow A_\infty(\modules \BbK),\]
defined on the level of objects by
\[ F(\EuP,X,Y) := \left(\EuP(X,Y),\mu_\EuP^{0|1|0}\right),\]
and on the level of morphisms as follows:
\begin{itemize}
\item For $(c_1,\ldots,c_s;d_1,\ldots,d_t) \in \EuC(X_1,\ldots,X_s) \otimes \EuD^{op}(Y_1,\ldots,Y_t)$, we define
\[ F^{0;s;t}(;c_1,\ldots,c_s;d_1,\ldots,d_t) \in \Hom_{\modules \BbK}(\EuP(X_s,Y_t),\EuP(X_1,Y_1))\]
to be the morphism which sends 
\begin{align*} p &\mapsto (-1)^\dagger \mu_\EuP(c_1,\ldots,c_s;p;d_t,\ldots,d_1),\quad\text{ for any $\EuP$, where}\\
 \dagger &:= \sum_{j<k} |d_j|'\cdot|d_k|' + |p| \cdot \sum_{j=1}^t |d_j|'.
 \end{align*}
\item For $(\rho;c_1,\ldots,c_s;d_1,\ldots,d_t) \in A_\infty([\EuC,\EuD])(\EuP_1,\EuP_2) \otimes \EuC(X_1,\ldots,X_s) \otimes \EuD^{op}(Y_1,\ldots,Y_t)$, we define
\[ F^{1;s;t}(\rho;c_1,\ldots,c_s;d_1,\ldots,d_t) \in \Hom_{\modules \BbK}(\EuP_2(X_s,Y_t),\EuP_1(X_1,Y_1))\]
to be the morphism which sends
\begin{align*} 
p &\mapsto (-1)^\dagger \rho(c_1,\ldots,c_s,p,d_t,\ldots,d_1),\quad\text{where}\\
 \dagger &:= |\rho| +  \sum_{j<k} |d_j|'\cdot|d_k|' + |p| \cdot \sum_{j=1}^t |d_j|'.
 \end{align*}
\end{itemize}
$F$ is strictly unital in its first entry.
\end{lem}

\begin{lem}
\label{lem:pullback}
Let $G: \EuC_1 \to \EuD_1$ and $H: \EuC_2 \to \EuD_2$ be $A_\infty$ functors. 
Then there is a $\mathsf{dg}$ functor
\[ (G \otimes H)^*: [\EuC_2,\EuD_2] \to [\EuC_1,\EuD_1].\]
It is given on the level of objects by defining $(G \otimes H)^* \EuP$ to be the $(\EuC_1,\EuD_1)$ bimodule with
\[ (G \otimes H)^*\EuP(X,Y) := \EuP(GX,HY),\]
\begin{multline}
\label{eqn:pullbackbimod}
 \mu^{s|1|t}_{(G \otimes H)^*\EuP}(c_1,\ldots,c_s,p,d_t,\ldots,d_1) := \\
 \sum \mu_\EuP^{*|1|*}(G(c_1,\ldots),\ldots,G(\ldots,c_s),p,H(d_t,\ldots),\ldots,H(\ldots,d_1)).
\end{multline}
It is given on the level of morphisms by mapping the bimodule pre-homomorphism $\rho$ to the bimodule pre-homomorphism $(G \otimes H)^*\rho$, given by the same formula \eqref{eqn:pullbackbimod}, but with `$\mu_\EuP$' replaced by `$\rho$'.
\end{lem}

\begin{lem}
\label{lem:pullbacktri}
Let $G: \EuC_1 \to \EuD_1$ and $H: \EuC_2 \to \EuD_2$ be $A_\infty$ functors, and denote by 
\[ F_i: A_\infty([\EuC_i,\EuD_i]) \times \EuC_i \times \EuD_i \dashrightarrow A_\infty(\modules \BbK)\]
the $A_\infty$ tri-functor introduced in Lemma \ref{lem:trifun}, for $i = 1,2$. 
Then we have an equality
\[ F_2 \circ (\mathsf{Id},G,H^{op}) = F_1 \circ (A_\infty((G \otimes H)^*),\mathsf{Id},\mathsf{Id})\]
 of $A_\infty$ tri-functors $A_\infty([\EuC_2,\EuD_2]) \times \EuC_1 \times \EuD_1^{op} \dashrightarrow A_\infty(\modules \BbK)$.
\end{lem}

\begin{lem}
\label{lem:multifunhoch}
Suppose that $\EuC_1,\ldots,\EuC_n$ are $A_\infty$ categories, $\EuD = A_\infty(\EuD')$ is the $A_\infty$ category corresponding to a $\mathsf{dg}$ category $\EuD'$, and $F: \EuC_1 \times \ldots \times \EuC_n \dashrightarrow \EuD$ is an $A_\infty$ $n$-functor. 
Then there is an induced chain map
\[ F_*: CC_\bullet(\EuC_1) \otimes \ldots \otimes CC_\bullet(\EuC_n) \to CC_\bullet(\EuD),\]
\begin{multline}
F_*(c_0^1[c_1^1|\ldots|c_{s_1}^1],\ldots,c_0^n[c_1^n|\ldots|c_{s_n}^n]) := \\
\sum (-1)^{\maltese+\dagger} \mu^2_\EuD\left(\ldots \left(\mu^2_\EuD \left(F^*\overbrace{(\ldots)}^1, F^*\overbrace{(\ldots)}^2 \right),\ldots\right),F^*\overbrace{(\ldots)}^n \right)
 \left[F^*(\ldots),\ldots,F^*(\ldots)\right].
\end{multline}
To clarify the notation: the first term is obtained by taking $n$ terms, and combining them with $n-1$ applications of $\mu^2_\EuD$ into a single term. 
The overbraces signify that we sum over all cyclic permutations of the inputs $c^i_j$ such that $c^i_0$ lands underneath the overbrace labelled $i$.
As usual, $\maltese$ is the Koszul sign associated to re-ordering the inputs $c^p_q$: this includes the Koszul signs associated with the cyclic re-ordering associated with the overbrace notation, exactly as in Section \ref{subsec:hochhom}.
The other contribution to the overall sign is
\[ \dagger := \frac{n(n-1)}{2} + \sum_{1 \le j \le n, 1 \le k \le s_j} (n-j) |c^j_k|'. \]
\end{lem}

\begin{lem}
\label{lem:compatcomp}
The maps induced by Lemma \ref{lem:multifunhoch} are compatible with composition of $A_\infty$ multifunctors, i.e., in the setting of Definition \ref{defn:multifuncomp}, we have $H_* = G_* \circ ((F_1)_* \otimes \ldots \otimes (F_n)_*)$.
\end{lem}

\begin{lem}
\label{lem:multifunsh}
In the situation of Example \ref{eg:dgmulti}, the diagram
\[ \xymatrix{ CC_\bullet(A_\infty(\EuC_1)) \otimes \ldots \otimes CC_\bullet(A_\infty(\EuC_n)) \ar[r]^-{=} \ar[d]^{F_*} & C_\bullet(\EuC_1) \otimes \ldots \otimes C_\bullet(\EuC_n)\ar[d]^{\mathsf{sh}} \\
CC_\bullet(A_\infty(\EuC_1 \otimes \ldots \otimes \EuC_n)) \ar[r]^-{=} &  C_\bullet(\EuC_1 \otimes \ldots \otimes \EuC_n)}\]
commutes. 
Here, $F_*$ is the map induced by the $A_\infty$ $n$-functor $F$ introduced in Example \ref{eg:dgmulti}, in accordance with Lemma \ref{lem:multifunhoch}. 
The other vertical map `$\mathsf{sh}$' is the natural generalization of the K\"{u}nneth quasi-isomorphism \eqref{eqn:kunndg}.
\end{lem}

\subsection{The Mukai pairing for $A_\infty$ bimodules}

\begin{defn}
\label{defn:chernchar}
If $X$ is an object of a cohomologically unital $A_\infty$ category $\EuC$, then the cohomological unit $e_X \in hom^\bullet(X,X)$ defines a Hochschild cycle; 
we call the corresponding class in Hochschild homology the \emph{Chern character of $X$}, and denote it $\mathsf{Ch}(X) \in \HH_0(\EuC)$.
\end{defn}

\begin{lem}
\label{lem:chernquasi}
If $X$ and $Y$ are quasi-isomorphic objects of $\EuC$, then $\mathsf{Ch}(X) = \mathsf{Ch}(Y)$. 
\end{lem}
\begin{proof}
Let $[f] \in \Hom^0(X,Y)$ and $[g] \in \Hom^0(Y,X)$ be inverse isomorphisms. Then
\[ \mu^2(f,g) = e_X + \mu^1(h_X),\,\,\, \mu^2(g,f) = e_Y + \mu^1(h_Y),\]
so
\[ b(f[g] - h_X + h_Y) = e_X - e_Y,\]
so the classes $[e_X]$ and $[e_Y]$ are cohomologous.
\end{proof}

\begin{defn}
\label{defn:wedgepainf}
Let $\EuC$ and $\EuD$ be $A_\infty$ categories, and denote by
\[ F_*: \HH_\bullet(A_\infty([\EuC,\EuD])) \otimes \HH_\bullet(\EuC) \otimes \HH_\bullet(\EuD^{op}) \to \HH_\bullet(\modules \BbK)\]
the map induced by the $A_\infty$ tri-functor $F$ introduced in Lemma \ref{lem:trifun}, in accordance with Lemma \ref{lem:multifunhoch}.  
We define the pairing
\begin{align*} \wedge_\EuP: \HH_\bullet(\EuC) \otimes \HH_\bullet(\EuD^{op}) &\to \HH_\bullet(\modules \BbK) \\
 \wedge_\EuP(\alpha,\beta) &:= F_*(\mathsf{Ch}(\EuP),\alpha,\beta).
 \end{align*}
\end{defn}

Definition \ref{defn:wedgepainf} is compatible with the corresponding notion in the $\mathsf{dg}$ world \eqref{eqn:wedgepdg}. 
To see how, we must first say how to turn a $\mathsf{dg}$ bimodule into an $A_\infty$ bimodule:

\begin{defn}
Let $\EuC$ and $\EuD$ be $\BbK$-linear $\mathsf{dg}$ categories, and $\EuP$ a $\mathsf{dg}$ $(\EuC,\EuD)$ bimodule. 
We define an $(A_\infty(\EuC),A_\infty(\EuD))$ bimodule $A_\infty(\EuP)$ with $A_\infty(\EuP)(X,Y) := \EuP(Y,X)$,
\[\mu^{0|1|0}  :=  d_\EuP;\qquad 
\mu^{1|1|0}(c,p)  :=  c \cdot p;\qquad
\mu^{0|1|1}(p,d) :=  (-1)^{|p|'} p \cdot d;\qquad
\mu^{s|1|t} := 0 \mbox{ for all $s+t \ge 2$.}  \]
\end{defn}

\begin{rmk}
\label{rmk:diags}
If $\EuP$ is the diagonal $(\EuC,\EuC)$ bimodule, then $A_\infty(\EuP)$ is tautologically isomorphic to the diagonal $(A_\infty(\EuC),A_\infty(\EuC))$ bimodule (as defined in \cite[Equation (2.20)]{Seidel2008c}).
\end{rmk}

\begin{lem}
\label{lem:dgainfcompat}
If $\EuC$ and $\EuD$ are $\mathsf{dg}$ categories, and $\EuP$ is a $\mathsf{dg}$ $(\EuC,\EuD)$ bimodule, then the diagram
\[
\xymatrix{ \HH_\bullet(\EuC) \otimes \HH_\bullet(\EuD^{op}) \ar[r] \ar[d]^{\wedge_\EuP} & \HH_\bullet(A_\infty(\EuC)) \otimes \HH_\bullet(A_\infty(\EuD)^{op}) \ar[d]^{\wedge_{A_\infty(\EuP)}}\\
\HH_\bullet(\modules \BbK) \ar[r] & \HH_\bullet(A_\infty(\modules \BbK))}
\]
commutes.
Here, the top arrow is the tautological isomorphism $\HH_\bullet(\EuC) \cong \HH_\bullet(A_\infty(\EuC))$, tensored with the isomorphism $\HH_\bullet(\EuD^{op}) \cong \HH_\bullet(A_\infty(\EuD)^{op})$ induced by the isomorphism of Remark \ref{rmk:minusop}.
\end{lem}
\begin{proof}
The diagram commutes on the level of cochain complexes: this follows by comparing Shklyarov's formula \eqref{eqn:Shkform} with our own definition.
\end{proof}

\begin{defn}
\label{defn:mukai}
Let $\EuC$ be a proper $A_\infty$ category. 
We define the \emph{Mukai pairing}
\begin{align}
\nonumber \langle ,\rangle_{Muk}: \HH_\bullet(\EuC) \otimes \HH_\bullet(\EuC) &\to \BbK\\
\label{eqn:mukaiainf}
 \langle \alpha,\beta \rangle_{Muk} &:= -\int \wedge_{\EuC_\Delta}(\alpha,\beta^\vee)
\end{align}
where $\EuC_\Delta$ is the diagonal bimodule, $\wedge$ is as in Definition \ref{defn:wedgepainf}, $\int$ denotes the Feigin--Losev--Shoikhet trace (which we can apply because $\EuC_\Delta$ is proper), and $\beta^\vee$ is the image of $\beta$ under the isomorphism of Definition \ref{defn:ccopp}.
\end{defn}

\begin{prop}
\label{prop:mukaimorita}
Let $\EuC$ and $\EuD$ be proper $A_\infty$ categories which are Morita equivalent. 
Then the isomorphism $\HH_\bullet(\EuC) \cong \HH_\bullet(\EuD)$ of Lemma \ref{lem:hochhommor} respects Mukai pairings.
\end{prop}
\begin{proof}
By Theorem \ref{thm:moritasplit}, it suffices to consider the case that the Morita equivalence is induced by a functor $F: \EuC \dashrightarrow \EuD$ which is cohomologically full and faithful and whose image split-generates. 
We will argue that
\[ \wedge_{\EuD_\Delta}(F_*\alpha,F_*\beta) = \wedge_{(F \otimes F)^* \EuD_\Delta}(\alpha,\beta) =\wedge_{\EuC_\Delta}.\]
The first equality follows by combining Lemma \ref{lem:pullbacktri} with Lemma \ref{lem:compatcomp}. 
To prove the second we observe that, because $F$ is cohomologically full and faithful, $(F \otimes F)^* \EuD_\Delta$ is quasi-isomorphic to $\EuC_\Delta$ in $[\EuC,\EuC]$ (cf. the proof of Lemma \ref{lem:moritafull}). 
Hence, by Lemma \ref{lem:chernquasi}, their Chern characters coincide, so the second equality is obvious from Definition \ref{defn:wedgepainf}. 
Composing with the Feigin--Losev--Shoikhet trace completes the proof.
\end{proof}

\begin{prop}
If $\EuC$ is a proper $\mathsf{dg}$ category, then our definition of the Mukai pairing on $\HH_\bullet(A_\infty(\EuC)) \cong \HH_\bullet(\EuC)$ (i.e., Definition \ref{defn:mukai}) coincides with that given by Shklyarov (i.e., Definition \ref{defn:dgmukai}).
\end{prop}
\begin{proof}
Follows immediately from Lemma \ref{lem:dgainfcompat}, together with Remarks \ref{rmk:diags} and \ref{rmk:veesign} (the discrepancy between the $\mathsf{dg}$ and $A_\infty$ versions of the isomorphism $\vee$ is the reason for the minus sign in \eqref{eqn:mukaiainf}).
\end{proof}

\begin{prop}
\label{prop:mukform}
If $\EuC$ is an $A_\infty$ category with finite-dimensional $hom$-spaces (i.e., finite-dimensional on the cochain level, not just on the cohomology level), then the Mukai pairing is induced by the following chain-level map: if $\alpha = a_0[a_1|\ldots|a_s]$ and $\beta = b_0[b_1|\ldots|b_t]$, then
\begin{align}
\label{eqn:mukform}
 \langle \alpha,\beta \rangle_{Muk} &= \sum_{j,k}\mathsf{tr} \left(c \mapsto (-1)^\dagger \mu^*(\overbrace{a_0,\ldots,a_j},\mu^*(a_{j+1},\ldots,a_s,c,\overbrace{b_0,\ldots,b_k}),b_{k+1},\ldots,b_t) \right),\\ 
\nonumber \text{ where}\quad \dagger &= 1+\sum_{i=0}^j |a_i|' + |c| \cdot |\beta|.
 \end{align}
To clarify \eqref{eqn:mukform}: if the expression is not composable in $\EuC$, we set the summand to be $0$. `$c$' represents an element in the corresponding $hom$-space of $\EuC$.
\end{prop}
\begin{proof}
By our assumption that $\EuC$ has finite-dimensional $hom$-spaces, $\wedge_{\EuC_\Delta}$ lands in $CC_\bullet(\fink \BbK)$. 
Therefore, we have
\[ \int \wedge_{\EuC_\Delta}(\alpha,\beta^\vee) = \mathsf{Str} \left(\wedge_{\EuC_\Delta}(\alpha,\beta^\vee)\right),\]
by Lemma \ref{lem:flsstr}. 
This yields \eqref{eqn:mukform}.
\end{proof}

\begin{example}
\label{eg:mukeul}
If $\EuC$ is an $A_\infty$ category with finite-dimensional $hom$-spaces, then \eqref{eqn:mukform} implies immediately that for any $X,Y \in Ob(\EuC)$, we have
\[ \langle \mathsf{Ch}(X),\mathsf{Ch}(Y) \rangle_{Muk} = \chi(\Hom^\bullet(X,Y))\]
(by applying the formula to $\alpha = e_X$ and $\beta = e_Y$, and observing that $\mu^2(e_X,\mu^2(a,e_Y)) = (-1)^{|a|} a$). 
Hence, the same holds for any proper $A_\infty$ category, by the homological perturbation lemma and Morita invariance. 
This is \cite[Theorem 1.3]{Shklyarov2012}.
\end{example}

We recall that an $A_\infty$ category $\EuC$ is called \emph{smooth} (or \emph{homologically smooth}) if the diagonal bimodule $\EuC_\Delta$ is \emph{perfect}, i.e., split-generated by tensor products of Yoneda modules (see \cite{Kontsevich2006a}). 
An $A_\infty$ category which is proper and smooth is called \emph{saturated}.

\begin{prop}
If $\EuC$ is saturated, then the Mukai pairing is non-degenerate.
\end{prop}
\begin{proof}
The result was proved for $\mathsf{dg}$ categories in \cite[Theorem 1.4]{Shklyarov2012}. 
Any $A_\infty$ category is quasi-equivalent to a $\mathsf{dg}$ category via the Yoneda embedding, so the result follows by Proposition \ref{prop:mukaimorita}.
\end{proof}

\subsection{Higher residue pairing on $\mathsf{dg}$ categories}

We recall the definition of the higher residue pairing given in \cite{Shklyarov2013}. 
If $\EuC$ and $\EuD$ are $\mathsf{dg}$ categories, then there is a K\"{u}nneth map of cochain complexes, extending \eqref{eqn:kunndg}:
\begin{equation}
\label{eqn:kunndgcyc}
C^-_\bullet(\EuC) \otimes C^-_\bullet(\EuD) \to C^-_\bullet(\EuC \otimes \EuD),
\end{equation}
and similarly for the other versions of cyclic homology (see \cite[Proposition 2.5]{Shklyarov2013}). 
This map induces an isomorphism on periodic cyclic homology, but need not induce an isomorphism on negative cyclic homology.

As before, a $\mathsf{dg}$ $(\EuC,\EuD)$ bimodule $\EuP$ induces a $\mathsf{dg}$ functor $\EuC \otimes \EuD^{op} \to \modules \BbK$; composing this with the K\"{u}nneth map \eqref{eqn:kunndgcyc} yields a map
\begin{equation}
\label{eqn:wedgetildecyc}
\tilde{\wedge}_\EuP: C_\bullet^- (\EuC) \otimes C_\bullet^-(\EuD^{op}) \to C_\bullet^-(\modules \BbK).
\end{equation}

Because the inclusion $\BbK \hookrightarrow \perfdg \BbK$ induces a quasi-isomorphism of Hochschild chain complexes, it also induces a quasi-isomorphism of cyclic homology complexes, by Corollary \ref{cor:cycmorita}.
We therefore obtain a quasi-isomorphism $C_\bullet^-(\BbK) \to C_\bullet^-(\perfdg \BbK)$. 
We know that $\HC^-_\bullet(\BbK) \cong \BbK\power{u}$ (c.f. \cite[(2.1.12)]{Loday1998}), so we obtain an isomorphism on the level of cohomology:
\begin{equation}
\label{eqn:FLScyc}
 \tildeint: \HC^-_\bullet(\perfdg \BbK) \to \BbK\power{u},
\end{equation}
the `cyclic Feigin--Losev--Shoikhet trace' (and similarly for periodic cyclic homology, where the map is to $\BbK\laurents{u}$, and positive cyclic homology, where the map is to $\BbK[u,u^{-1}]/\BbK[u]$).

The $\BbK\power{u}$-linear extension of the map $\mathsf{Str}$ defined in Lemma \ref{lem:flsstr} defines a chain map
\[ \widetilde{\mathsf{Str}}: CC^-_\bullet(\fink \BbK) \to \BbK\power{u}.\]
The same argument as given in the proof of Lemma \ref{lem:flsstr} shows that the induced map on the level of cohomology coincides with the map
\[ \HC^-_\bullet(\fink \BbK) \to \HC^-_\bullet(\perfdg \BbK) \overset{\overset{\,\,\sim}{\int}}{\to} \BbK\power{u}.\]

\begin{defn}
\label{defn:dghigherres}
If $\EuC$ is a proper $\mathsf{dg}$ category, we define the \emph{higher residue pairing}, which is the pairing
\begin{align*}
 \langle ,\rangle_{res}: \HC^-_\bullet(\EuC) \times \HC^-_\bullet(\EuC) &\to \BbK\power{u}\\
 \langle \alpha, \beta \rangle_{res} &:= \tildeint \tilde{\wedge}_{\EuC_\Delta}(\alpha,\beta^\vee).
 \end{align*}
The pairing is sesquilinear.
We obtain similar pairings on $\HP_\bullet$ and $\HC^+_\bullet$.
\end{defn}

\begin{rmk}
Because the K\"{u}nneth quasi-isomorphism for negative cyclic homology \eqref{eqn:kunndgcyc} extends that for Hochschild homology \eqref{eqn:kunndg}, and because the cyclic Feigin--Losev--Shoikhet trace \eqref{eqn:FLScyc} extends the non-cyclic version \eqref{eqn:FLS}, the higher residue pairing extends the Mukai pairing, in the sense that
\[ G\langle \alpha, \beta \rangle_{res} = \langle G\alpha, G\beta \rangle_{Muk},\]
where on the left-hand side, $G:\BbK\power{u} \to \BbK$ is the map setting $u=0$, and on the right-hand side, $G: \HC^-_\bullet \to \HH_\bullet$ is the map induced on Hochschild complexes.
\end{rmk}

\subsection{Higher residue pairing for $A_\infty$ bimodules}

\begin{defn}
\label{defn:F'}
Let $F: \EuC_1 \times \EuC_2 \times \EuC_3 \dashrightarrow \EuD$ be an $A_\infty$ tri-functor, where $\mu^{\ge 3}_\EuD = 0$. 
We define a $\BbK\power{u}$-linear map 
\[ F'_*: CC_\bullet^-(\EuC_1) \otimes CC_\bullet^-(\EuC_2) \otimes CC_\bullet^-(\EuC_3) \to CC_\bullet^-(\EuD)\]
as a sum of three maps: $F'_* := F1+F2+F3$.
For $\alpha= a_0[a_1|\ldots|a_s]$, $\beta = b_0[b_1|\ldots|b_t]$, $\gamma = c_0[c_1|\ldots|c_u]$, we define
\begin{multline}
F1(\alpha,\beta,\gamma) := \\
\sum (-1)^{\maltese+\dagger+|\beta|'} e^+[F^*(a_0,\ldots;b_0,\ldots;c_0,\ldots)|\ldots|\mu^2_\EuD(F^*(\overbrace{\vphantom{l}\ldots}^1),F^*(\overbrace{\vphantom{l}\ldots}^2))|F^*(\ldots)|\ldots|F^*(\overbrace{\vphantom{l}\ldots}^3)|\ldots|F^*(\ldots)],
\end{multline}
where $\maltese$ is Koszul sign associated to re-ordering the inputs $a_i,b_j,c_k$ as before (ignoring $e^+$), and $\dagger$ is the Koszul sign associated to commuting $\mu^2_\EuD$ (equipped with sign $1$) to the front of the expression, where all $F^*$ have degree $0$, all $a_i,b_j,c_k$ have their reduced degrees, and $e^+$ has degree $0$.

We define
\begin{multline}
F2(\alpha,\beta,\gamma) := \\
\sum (-1)^{\maltese + |\beta|} F^*(\overbrace{\vphantom{l}a_0,\ldots;b_0,\ldots;c_0,\ldots}^3)[F^*(\ldots)|\ldots|F(\overbrace{\vphantom{l}\ldots}^1)|\ldots|F(\overbrace{\vphantom{l}\ldots}^2)|\ldots|F^*(\ldots)].
\end{multline}

We define
\begin{multline}
F3(\alpha,\beta,\gamma) := 
\sum (-1)^{\maltese + |\beta|} F^*(\overbrace{\vphantom{l}a_0,\ldots;b_0,\ldots;c_0,\ldots}^{1,3})[F^*(\ldots)|\ldots|F(\overbrace{\vphantom{l}\ldots}^2)|\ldots|F^*(\ldots)].
\end{multline}
\end{defn}

\begin{lem}
\label{lem:multifuncyc}
Let $F: \EuC_1 \times \EuC_2 \times \EuC_3 \dashrightarrow \EuD$ be an $A_\infty$ tri-functor, where $\mu^{\ge 3}_\EuD = 0$. 
Then there is a $\BbK\power{u}$-linear chain map
\[ \widetilde{F}_*: CC_\bullet^-(\EuC_1) \otimes CC_\bullet^-(\EuC_2) \otimes CC_\bullet^-(\EuC_3) \to CC_\bullet^-(\EuD)\]
(the tensor products are over $\BbK\power{u}$), defined by
\[ \widetilde{F}_* := F_* + uF'_*,\]
where $F_*$ is as in Lemma \ref{lem:multifunhoch} (in the case $n=3$), and $F'_*$ is as in Definition \ref{defn:F'}. 
The analogous results also hold for the periodic and positive versions of cyclic homology.
\end{lem}
\begin{proof}
In order to prove that
\[ \widetilde{F}_* \circ (b+uB) = (b+uB)  \circ \widetilde{F}_*,\]
it suffices to prove that: $F_* \circ b = b \circ F_*$ (we proved this in Lemma \ref{lem:multifunhoch});
$F'_* \circ b + F_* \circ B = b \circ F'_* + B \circ F_*$; and $F'_* \circ B = B \circ F'_*$.
Each of these is a trivial check using the graphical notation of Appendix \ref{app:signs}; we omit the details.
\end{proof}

\begin{defn}
\label{defn:cycchern}
As in Definition \ref{defn:chernchar}, any object $X$ of an $A_\infty$ category $\EuC$ has an associated `cyclic Chern character' $ \widetilde{\mathsf{Ch}}(X) \in \HC_0^-(\EuC)$ 
(similarly for periodic and positive versions). 
Quasi-isomorphic objects have the same cyclic Chern character. 
If $\EuC$ is strictly unital, the Chern character has a particularly simple cochain-level representative. 
Namely, in the presence of strict units, we can define the Connes differential (and all other operations we have considered so far) using the strict units $e$ in place of $e^+$. 
This gives the \emph{unital cyclic complex} $ ((CC_\bullet/D_\bullet) \widehat{\otimes} W, b+uB)$.  
In the unital cyclic complex, $\widetilde{\mathsf{Ch}}(X)$ is represented on the cochain level by the length-$0$ cycle $e_X$.
\end{defn}

\begin{rmk}
We do not give a proof of the assertions made in Definition \ref{defn:cycchern}, but they are standard: in fact, one can show that $\widetilde{\mathsf{Ch}}(X)$ depends only on the class of $[X] \in K_0(\EuC)$ (cf. \cite[Lemma 8.4]{Seidel2013}).
\end{rmk}

\begin{defn}
\label{defn:ainfcycwedge}
Let $\EuP$ be an $A_\infty$ $(\EuC,\EuD)$ bimodule. 
We define a $\BbK\power{u}$-linear pairing
\begin{eqnarray*} 
\widetilde{\wedge}_\EuP: \HC_\bullet^-(\EuC) \otimes_{\BbK\power{u}} \HC_\bullet^-(\EuD^{op}) &\to \HC^-_\bullet(A_\infty(\modules \BbK)),\\
 \widetilde{\wedge}_\EuP(\alpha,\beta) &:= \widetilde{F}_*\left(\widetilde{\mathsf{Ch}}(\EuP),\alpha,\beta \right),
 \end{eqnarray*}
where $F: A_\infty([\EuC,\EuD]) \times \EuC \times \EuD^{op} \dashrightarrow A_\infty(\modules \BbK)$ is the $A_\infty$ tri-functor of Lemma \ref{lem:trifun}, and $\widetilde{F}_*$ is the induced map, in accordance with Lemma \ref{lem:multifuncyc}.
\end{defn}

\begin{lem}
\label{lem:dgainfcyccompat}
Definition \ref{defn:ainfcycwedge} is compatible with the corresponding definition in the $\mathsf{dg}$ world, i.e., the following diagram commutes:
\[ \xymatrix{ \HC^-_\bullet(\EuC) \otimes \HC^-_\bullet(\EuD^{op}) \ar[r] \ar[d]^{\widetilde{\wedge}_\EuP} & \HC^-_\bullet(A_\infty(\EuC)) \otimes \HC^-_\bullet(A_\infty(\EuD)^{op}) \ar[d]^{\widetilde{\wedge}_{A_\infty(\EuP)}} \\
\HC^-_\bullet(\modules \BbK) \ar[r] & \HC^-_\bullet(A_\infty(\modules \BbK)).}\]
Here, the horizontal arrows are the tautological identifications (or the isomorphism induced by the the isomorphism of Remark \ref{rmk:minusop}, in the case of $\EuD^{op}$). 
The left vertical map is the map \eqref{eqn:wedgetildecyc}, and the right vertical map is the map introduced in Definition \ref{defn:ainfcycwedge}.
\end{lem}
\begin{proof}
The diagram commutes on the level of cochain complexes.
\end{proof}

\begin{defn}
\label{defn:ainfhigherres}
Let $\EuC$ be a proper $A_\infty$ category. 
We define the \emph{higher residue pairing}
\begin{align}\nonumber \langle ,\rangle_{res}: \HC^-_\bullet(\EuC) \times \HC^-_\bullet(\EuC) &\to \BbK\power{u}\\
\label{eqn:resainf}
 \langle \alpha,\beta \rangle_{res} &:= -\tildeint \widetilde{\wedge}_{\EuC_\Delta}(\alpha,\beta^\vee)
\end{align}
where $\EuC_\Delta$ is the diagonal bimodule, $\widetilde{\wedge}$ is as in Definition \ref{defn:ainfcycwedge}, $\tildeint$ denotes the cyclic Feigin--Losev--Shoikhet trace (which we can apply because $\EuC_\Delta$ is proper), and $\beta^\vee$ is the image of $\beta$ under the isomorphism of Remark \ref{rmk:cycopp}. 
The higher residue pairing is $\BbK\power{u}$-sesquilinear.
\end{defn}

\begin{prop}[Morita invariance of higher residue pairing]
\label{prop:higherresMor}
Let $\EuC$ and $\EuD$ be proper $A_\infty$ categories, which are Morita equivalent. 
Then the isomorphism $\HC_\bullet^-(\EuC) \cong \HC_\bullet^-(\EuD)$ of Corollary \ref{cor:cycGMmor} respects higher residue pairings.
\end{prop}
\begin{proof}
The proof follows that of Proposition \ref{prop:mukaimorita} closely.
\end{proof}

\begin{prop}
If $\EuC$ is a $\mathsf{dg}$ category, then our definition of the higher residue pairing on $\HC^-_\bullet(A_\infty(\EuC)) \cong \HC^-_\bullet(\EuC)$ (i.e., Definition \ref{defn:ainfhigherres}) coincides with that given by Shklyarov (i.e., Definition \ref{defn:dghigherres}).
\end{prop}
\begin{proof}
Follows immediately from Lemma \ref{lem:dgainfcyccompat}, together with Remarks \ref{rmk:diags} and \ref{rmk:veesign}.
\end{proof}

\begin{prop}
\label{prop:higherresform}
If $\EuC$ is an $A_\infty$ category with finite-dimensional $hom$-spaces (i.e., finite-dimensional on the cochain level, not just on the cohomology level), then the higher residue pairing is induced by a chain-level map, given by extending the formula \eqref{eqn:mukform} $\BbK\power{u}$-sesquilinearly.
\end{prop}
\begin{proof}
Observe that $[\EuC,\EuC]$ is strictly unital, so we have the explicit representing cycle $e_{\EuC_\Delta}$ for $\widetilde{\mathsf{Ch}}(\EuC_\Delta)$, as explained in Definition \ref{defn:cycchern}.
We now observe that $\widetilde{F}_* = F_* + uF'_*$, so on the chain level we have
\[ \langle \alpha, \beta \rangle_{res}  =   \mathsf{Str} \left(F_*(e_{\EuC_\Delta},\alpha,\beta^\vee)\right) + u \mathsf{Str} \left(F'_*(e_{\EuC_{\Delta}},\alpha,\beta^\vee)\right). \]
Now observe that $F_*'$ never outputs a term of length $1$ (see Definition \ref{defn:F'}), so $\mathsf{Str} \circ F'_* = 0$ on the chain level. 
The proof now follows from that of Proposition \ref{prop:mukform}. 
\end{proof}

\subsection{The higher residue pairing is covariantly constant}

\begin{defn}
Let $\EuC$ be a $\BbK$-linear $A_\infty$ category with finite-dimensional $hom$-spaces, together with a choice of basis for each $hom$-space (which we recall is necessary to make sense of expressions like `$v(\mu^*)$'). 
For each derivation $v \in \deriv_\Bbbk \BbK$, we define a $\BbK\power{u}$-sesquilinear map
\[ H: CC^-_\bullet(\EuC) \times CC^-_\bullet(\EuC) \to \BbK\power{u}\]
as a sum of three terms: $H:= H1+H2+H3$.
For $\alpha= a_0[a_1|\ldots|a_s]$ and $\beta = b_0[b_1|\ldots|b_t]$, we define
\begin{multline}
H1(\alpha,\beta) := \\
\sum_{j,k,\ell,m}\mathsf{tr} \left(c \mapsto (-1)^\dagger \mu^*(a_0,\ldots,v(\mu^*)(a_{j+1},\ldots),\overbrace{a_{k+1},\ldots,a_\ell},\mu^*(a_{\ell+1},\ldots,a_s,c,\overbrace{b_0,\ldots,b_m}),b_{m+1},\ldots,b_t) \right),
\end{multline}
\[\text{where}\quad \dagger:= \sum_{i=j+1}^\ell |a_i|' + |c| \cdot |\beta|;\]
\begin{multline}
H2(\alpha,\beta) = \\
\sum_{j,k,\ell,m}\mathsf{tr} \left(c \mapsto (-1)^\dagger \mu^*(\overbrace{a_0,\ldots,a_j},\mu^*(a_{j+1},\ldots,a_s,c,b_0,\ldots,v(\mu^*)(b_{k+1},\ldots),\overbrace{b_{\ell+1},\ldots}),b_{m+1},\ldots,b_t) \right),
\end{multline}
\[ \text{where}\quad \dagger := 1+\sum_{i=j+1}^s |a_i|' +  \sum_{i=0}^k |b_i|' + |c| \cdot |\beta|';\quad\text{and}\]
\begin{multline}
H3(\alpha,\beta) := \\
\sum_{j,k}\mathsf{tr} \left(c \mapsto (-1)^\dagger v(\mu^*)(a_0,\ldots,\mu^*(\overbrace{a_{j+1},\ldots},\mu^*(a_{k+1},\ldots,a_s,c,\overbrace{b_0,\ldots}),b_{\ell+1},\ldots),b_{m+1},\ldots,b_t) \right),
\end{multline}
\[ \text{where}\quad \dagger := \sum_{i=j+1}^k |a_i|' + |c| \cdot |\beta|.\]
\end{defn}

\begin{lem}
\label{lem:higherreshomflat}
We have
\[ \langle u\cdot \nabla_v \alpha,\beta \rangle_{res} - \langle \alpha, u\cdot \nabla_v \beta \rangle_{res} = u\cdot v \langle \alpha,\beta \rangle_{res} + H \circ (b + uB)\]
as $\BbK\power{u}$-sesquilinear maps from $CC_\bullet^-(\EuC) \times CC_\bullet^-(\EuC) \to \BbK\power{u}$. 
\end{lem}

\begin{cor}
\label{cor:hrescov}
Let $\EuC$ be a proper $\BbK$-linear $A_\infty$ category. 
Then the higher residue pairing is covariantly constant with respect to the Getzler--Gauss--Manin connection: i.e., for any $v \in \deriv_\Bbbk \BbK$, we have
\[ \langle u \cdot \nabla_v \alpha,\beta \rangle_{res} - \langle \alpha, u \cdot \nabla_v \beta \rangle_{res} = u \cdot v \langle \alpha,\beta \rangle_{res}\]
as $\BbK\power{u}$-sesquilinear maps from $\HC^-_\bullet(\EuC) \times \HC^-_\bullet(\EuC) \to \BbK\power{u}$.
\end{cor}
\begin{proof}
By the homological perturbation lemma, any $A_\infty$ category $\EuC$ is quasi-isomorphic to a minimal $A_\infty$ category $\EuC'$ (i.e., one with $\mu^1 = 0$). 
We have $\HC^-_\bullet(\EuC) \cong \HC^-_\bullet(\EuC')$, and the isomorphism respects connections (Theorem \ref{thm:GMMorita}) and higher residue pairings (Proposition \ref{prop:higherresMor}), so it suffices to prove the result for $\EuC'$. 
Because $\EuC$ is proper, $\EuC'$ will have finite-dimensional $hom$-spaces, so its higher residue pairing is covariantly constant by Lemma \ref{lem:higherreshomflat}.
\end{proof}

\subsection{Symmetry}\label{subsec:sym}

Let $\EuP$ be a $(\EuC,\EuD)$ bimodule. 
We recall the definition of the \emph{shift}, $\EuP[1]$, which is a $(\EuC,\EuD)$ bimodule with $\EuP[1](X,Y) = \EuP(X,Y)[1]$ 
(see \cite[Equation (2.10)]{Seidel2008c}). 

\begin{lem}
\label{lem:shiftpair}
Let $\EuP$ be a proper $(\EuC,\EuD)$ bimodule. 
Then 
\[ \int \wedge_{\EuP[1]}(\alpha,\beta) = -\int \wedge_{\EuP}(\alpha,\beta).\]
\end{lem}
\begin{proof}
We may assume that $\EuC$ and $\EuD$ are $\mathsf{dg}$ categories, and $\EuP$ a proper $\mathsf{dg}$ bimodule.
We recall that $\EuP$ corresponds to a $\mathsf{dg}$ functor $\EuP: \EuC \otimes \EuD^{op} \to \perfdg \BbK$, and that Shklyarov \cite{Shklyarov2012} defines $\wedge_\EuP$ to be the map induced by this functor on Hochschild homology, composed with the K\"{u}nneth isomorphism. 
It is clear that $\EuP[1]$ corresponds to the composition of $\mathsf{dg}$ functors $S \circ \EuP$, where $S: \perfdg \BbK \to \perfdg \BbK$ is the shift functor. 
It now suffices to check that $\int S_*(\alpha) = -\int\alpha$, where $S_*: \HH_\bullet(\perfdg \BbK) \to \HH_\bullet(\perfdg \BbK)$ is the map induced by the functor $S$. 
Because the inclusion $\fink \BbK \hookrightarrow \perfdg \BbK$ is a quasi-equivalence, it suffices by Lemma \ref{lem:flsstr} to prove that $\mathsf{Str} (S_*(\alpha)) = -\mathsf{Str}(\alpha)$ on $\HH_\bullet(\fink \BbK)$. 
This is clear from the definition of the supertrace.
\end{proof}

We also recall the \emph{linear dual} bimodule, $\EuP^\vee$, which is a $(\EuD,\EuC)$ bimodule with  $\EuP^\vee(Y,X) = \hom(\EuP(X,Y),\BbK)$ (see, e.g., \cite[Equation (2.11)]{Seidel:II}).

\begin{lem}
\label{lem:dualpair}
Let $\EuP$ be a proper $(\EuC,\EuD)$ bimodule. 
Then
\[ \int \wedge_{\EuP^\vee}(\alpha,\beta) = (-1)^{|\alpha| \cdot |\beta|} \int \wedge_{\EuP}(\beta^\vee,\alpha^\vee).\]
\end{lem}
\begin{proof}
Once again, we may assume that $\EuC$, $\EuD$ and $\EuP$ are $\mathsf{dg}$, and regard $\EuP$ as a $\mathsf{dg}$ functor. 
The proof combines four pieces. First, let $\EuP^\vee$ denote the following composition of $\mathsf{dg}$ functors:
\[ \EuD \otimes \EuC^{op} \overset{\sim}{\longrightarrow} \left(\EuC \otimes \EuD^{op}\right)^{op} \overset{\EuP^{op}}{\longrightarrow} (\perfdg \BbK)^{op} \overset{\mathsf{dual}}{\longrightarrow} \perfdg \BbK,\]
where the first functor sends $c \otimes d \mapsto (-1)^{|c| \cdot |d|} d \otimes c$ and `$\mathsf{dual}$' denotes the $\mathsf{dg}$ functor that dualizes cochain complexes. 
One easily verifies that this is compatible with the $A_\infty$ definition, in the sense that $ A_\infty(\EuP)^\vee \cong A_\infty(\EuP^\vee)$.

Second, for any $\mathsf{dg}$ functor $F$, one easily checks that $(F^{op})_* (\alpha^\vee) = F_*(\alpha)^\vee$ (we apply this to $F = \EuP$).

Third, one checks that the following diagram commutes:
\[ \xymatrix{ CC_\bullet(\EuC) \otimes CC_\bullet(\EuD^{op}) \ar[r] \ar[d] & CC_\bullet(\EuC \otimes \EuD^{op}) \ar[d]\\
CC_\bullet(\EuD) \otimes CC_\bullet(\EuC^{op}) \ar[r] & CC_\bullet(\EuD \otimes \EuC^{op}),}\]
where the horizontal arrows are the K\"{u}nneth maps \cite[Section 2.4]{Shklyarov2012}, the left vertical arrow combines the isomorphism $CC_\bullet(\EuC) \cong CC_\bullet(\EuC^{op})$ of \cite[Equation (4.8)]{Shklyarov2012} (and similarly for $\EuD$) with the Koszul isomorphism $CC_\bullet(\EuC) \otimes CC_\bullet(\EuD^{op}) \cong  CC_\bullet(\EuD^{op}) \otimes CC_\bullet(\EuC)$,  and the right vertical arrow is induced by the isomorphism $CC_\bullet(\EuC \otimes \EuD^{op}) \cong CC_\bullet((\EuC \otimes \EuD^{op})^{op})$ composed with the map induced by the isomorphism of $\mathsf{dg}$ categories, $(\EuC \otimes \EuD^{op})^{op} \cong \EuC^{op} \otimes \EuD \cong \EuD \otimes \EuC^{op}$.

Fourth, one checks that $\int \mathsf{dual}_*(\alpha^\vee) = \int \alpha$. 
As in the proof of Lemma \ref{lem:shiftpair}, it suffices to prove that $\mathsf{Str}(\mathsf{dual}_*(\alpha^\vee)) = \mathsf{Str}(\alpha)$; and this reduces to the obvious fact that the trace of the dual of a matrix coincides with the trace of the original matrix. 
Combining the four pieces gives the result.
\end{proof}
 
We recall that an \emph{$n$-dimensional weak proper Calabi--Yau structure} on an $A_\infty$ category $\EuC$ is a quasi-isomorphism $\EuC_\Delta \cong \EuC_\Delta^\vee[n]$ (see \cite[Section 12j]{Seidel2008}, as well as \cite{Tradler2008} and \cite[Section A.5]{Sheridan:Fano}, where it is called an `$\infty$-inner product'). 

\begin{lem}
\label{lem:nsym}
If $\EuC$ admits an $n$-dimensional weak proper Calabi--Yau structure, then the Mukai pairing on $\HH_\bullet(\EuC)$ satisfies:
\[\langle \alpha, \beta\rangle_{Muk} = (-1)^{n+|\alpha|\cdot|\beta|} \langle \beta, \alpha \rangle_{Muk}.\]
Similarly for the higher residue pairing.
\end{lem}
\begin{proof}
The result for the the Mukai pairing follows directly from Lemmas \ref{lem:shiftpair} and \ref{lem:dualpair}, and the fact that the pairing $\wedge_\EuP$ only depends on the quasi-isomorphism class of the bimodule $\EuP$. 
The proof for the higher residue pairing is analogous.
\end{proof}

Lemma \ref{lem:nsym} completes the proof of Theorem \ref{thm:main} \eqref{it:polprevshs}.

\subsection{Hodge-to-de Rham degeneration}

\begin{defn}
\label{defn:HdR}
Suppose that $\EuC$ is saturated. 
We say that $\EuC$ satisfies the \emph{degeneration hypothesis} if the spectral sequence \eqref{eqn:HdR} induced by the Hodge filtration on $CC_\bullet^-(\EuC)$ degenerates at the $E_1$ page. 
\end{defn}

\begin{rmk}
A conjecture of Kontsevich and Soibelman \cite[Conjecture 9.1.2]{Kontsevich2006a} asserts that all saturated $A_\infty$ categories $\EuC$ satisfy the degeneration hypothesis. 
It has been proved by Kaledin \cite{Kaledin2017} (see also \cite{Mathew2017}), in the case that $\EuC$ is $\Z$-graded.
\end{rmk}

\begin{thm}[= Theorem \ref{thm:main} \eqref{it:degen}]
\label{thm:degen}
If $\EuC$ is saturated, and satisfies the degeneration hypothesis, then the polarized pre-\vshs{} $(\HC_\bullet^-(\EuC),\nabla,\langle\cdot,\cdot\rangle_{res})$ of Theorem \ref{thm:main} \eqref{it:polprevshs} is a polarized \vshs.
\end{thm}
\begin{proof}
Any $A_\infty$ category is quasi-equivalent to a $\mathsf{dg}$ category, via the Yoneda embedding; so let us assume without loss of generality that $\EuC$ is $\mathsf{dg}$. 
Then $\HH_\bullet(\EuC)$ is finite-dimensional and the Mukai pairing is non-degenerate, by \cite[Theorem 1.4]{Shklyarov2012}; it follows that the polarization given by the higher residue pairing is non-degenerate. 

As an immediate consequence, the spectral sequence \eqref{eqn:HdR} induced by the Hodge filtration on any of $CC_\bullet^{+,-,\infty}(\EuC)$ is automatically bounded (by the degree bound on $\HH_\bullet$) and regular (by finite-dimensionality); because the Hodge filtration is complete and exhaustive, the complete convergence theorem \cite[Theorem 5.5.10]{Weibel1994}  implies that the spectral sequence converges to its cohomology.
Hence, $\HC_\bullet^-(\EuC)$ has finite rank.

Finally, if $\EuC$ satisfies the degeneration hypothesis, then it is clear that $\HC^-_\bullet(\EuC)$ is a free $\BbK\power{u}$-module. 
Thus we have verified all of the conditions of Definitions \ref{defn:vhs} and \ref{defn:polvhs}, so $(\HC_\bullet^-(\EuC),\nabla,\langle\cdot,\cdot\rangle_{res})$ is a polarized \vshs. 
\end{proof}

\appendix

\section{Morita equivalence}
\label{app:morita}

In this Appendix we provide proofs of some well-known results relating to Morita equivalence of $A_\infty$ categories. 

\begin{lem}
\label{lem:bimodsplit}
Let $\rho: \EuP \to \EuQ$ be a homomorphism of $(\EuC,\EuD)$ bimodules. 
Let $\EuX \subset Ob(\EuC)$ and $\EuY \subset Ob(\EuD)$ be subsets which split-generate $\EuC$ and $\EuD$ respectively. 
If the map
\begin{equation}
\label{eqn:rhomod} \rho^{0|1|0}: \EuP(X,Y) \to \EuQ(X,Y)
\end{equation}
is a quasi-isomorphism for all $(X,Y) \in \EuX \times \EuY$, then $\rho$ is a quasi-isomorphism.
\end{lem}
\begin{proof}
Denote by $\EuS \subset Ob(\EuC)$ the set of objects $X$ such that \eqref{eqn:rhomod} is a quasi-isomorphism for all $Y \in \EuY$. This set contains $\EuX$ by hypothesis, and it is straightforward to show that it is closed under forming cones and direct summands; therefore it is all of $Ob(\EuC)$ because $\EuX$ split-generates. 
Now repeat the argument for the set $\EuT \subset Ob(\EuD)$ of objects $Y$ such that \eqref{eqn:rhomod} is a quasi-isomorphism for all $X \in Ob(\EuC)$.
\end{proof}

\begin{lem}[=Lemma \ref{lem:moritasplit}]
\label{lem:moritafull}
If $F: \EuC \to \EuD$ is a cohomologically full and faithful $A_\infty$ functor, and $\EuD$ is split-generated by the image of $F$, then $\EuM := (F \otimes \mathrm{Id})^* \EuD_\Delta$ and $\EuN := (\mathrm{Id} \otimes F)^* \EuD_\Delta$ define a Morita equivalence between $\EuC$ and $\EuD$.
\end{lem}
\begin{proof}
Tensor products of bimodules respect pullbacks in the following sense: If $F_i: \EuC_i \to \EuD_i$ are $A_\infty$ functors for $i = 1,2,3$, then there is a morphism of bimodules
\begin{equation}
\label{eqn:tensfunc}
 (F_1 \otimes F_2)^* \EuM \otimes_{\EuC_2} (F_2 \otimes F_3)^* \EuN \to (F_1 \otimes F_3)^* \left(\EuM \otimes_{\EuD_2} \EuN \right).
\end{equation}
It is given by the formula
\[ m [ a_1 | \ldots |a_s] n \mapsto \sum m[ F_2(a_1,\ldots)|\ldots| F_2(\ldots,a_s)] n\]
(no higher maps). 
If $F_2$ is the identity functor, it is clear from the formula that \eqref{eqn:tensfunc} is the identity. 
It follows immediately that $ \EuM \otimes_\EuD \EuN \cong (F \otimes F)^* \EuD_\Delta$. Now there is a natural morphism $ \EuC_\Delta \to (F \otimes F)^* \EuD_\Delta$,
given by contracting all terms with $F$. 
This morphism is clearly a quasi-isomorphism when $F$ is cohomologically full and faithful. 
Therefore, $\EuC_\Delta \cong \EuM \otimes_\EuD \EuN$ as required.

It remains to prove that $\EuN \otimes_\EuC \EuM \cong \EuD_\Delta$. 
From \eqref{eqn:tensfunc}, we obtain a morphism of bimodules
\begin{equation}
\label{eqn:perftens}
 \EuN \otimes_\EuC \EuM \to (\mathrm{Id} \otimes \mathrm{Id})^* (\EuD_\Delta \otimes_\EuD \EuD_\Delta) \overset{\cong}{\to} \EuD_\Delta. 
\end{equation}
So it remains to prove that this morphism is a quasi-isomorphism of $(\EuD,\EuD)$ bimodules. 
The linear term of \eqref{eqn:perftens} is the map
\begin{equation}
\label{eqn:rhofirst}
\bigoplus hom^\bullet_\EuD(F(X_0),U) \otimes hom^\bullet_\EuC(X_1,X_0) \otimes \ldots \otimes hom^\bullet_\EuC(X_s,X_{s-1}) \otimes hom^\bullet_\EuD(V,F(X_s)) \to hom^\bullet_\EuD(V,U) 
\end{equation}
given by the formula
\begin{equation}
\label{eqn:rhofirstform}
m[a_1|\ldots |a_s] n  \mapsto  \sum \mu^*(m,F^\bullet(c_1,\ldots), \ldots,F^\bullet(\ldots,a_s),n).
\end{equation}

We now prove that \eqref{eqn:rhofirst} is a quasi-isomorphism in the special case that $U = F(\tilde{U})$ and $V= F(\tilde{V})$. 
To do this, we observe that the following diagram commutes up to homotopy:
\begin{equation}
\label{eqn:splitF} \xymatrix{ \bigoplus hom^\bullet_\EuC(X_0,\tilde{U}) \otimes \EuC(X_0,\ldots,X_s) \otimes hom^\bullet_{\EuC}(\tilde{V},X_s) \ar[d]^-{F^\bullet \otimes \mathrm{Id} \otimes F^\bullet} \ar[r]^-{\mu^*_\EuC} & hom^\bullet_\EuC(\tilde{U},\tilde{V}) \ar[d]^{F^1} \\
\bigoplus hom^\bullet_\EuD(F(Y_0),F(\tilde{U})) \otimes \EuC(Y_0,\ldots,Y_t ) \otimes hom^\bullet_\EuD(F(\tilde{V}),F(Y_t)) \ar[r] & hom^\bullet_{\EuD}(F(\tilde{U}),F(\tilde{V})).} 
\end{equation}
The left vertical map sends
\[ m[ a_1| \ldots|a_s ] n  \mapsto  \sum F^\bullet(m,a_1,\ldots)[a_{i+1}| \ldots| a_j ]F^\bullet(a_{j+1},\ldots,a_s,n).\]
(compare \cite[Equation (2.240)]{Ganatra2013}). 
The  bottom horizontal map is precisely the map \eqref{eqn:rhofirst}.
The diagram commutes up to the homotopy given by
\[ m[a_1|\ldots |a_s] n \mapsto F^\bullet(m,c_1,\ldots,c_s,n),\]
 using the $A_\infty$ functor equations for $F$. 
Furthermore, the top horizontal arrow is a quasi-isomorphism (it is the first term in the quasi-isomorphism $\EuC_\Delta \otimes_\EuC \EuC_\Delta \cong \EuC_\Delta$); the right vertical arrow is a quasi-isomorphism (because $F$ is cohomologically full and faithful); the left vertical arrow is a quasi-isomorphism, by a comparison argument for the spectral sequences induced by the obvious length filtrations, again using the fact that $F$ is cohomologically full and faithful. 
It follows that the bottom map is a quasi-isomorphism. 
So \eqref{eqn:rhofirst} is a quasi-isomorphism when $U$ and $V$ are in the image of $F$. 

It follows that \eqref{eqn:rhofirst} is a quasi-isomorphism of bimodules, by Lemma \ref{lem:bimodsplit} and the hypothesis that the image of $F$ split-generates $\EuD$.
\end{proof}

\begin{thm}
\label{thm:moritaonlyif}
If $\EuC$ and $\EuD$ are Morita equivalent, then $\twsplit \EuC$ and $\twsplit \EuD$ are quasi-equivalent.
\end{thm}
\begin{proof}
This is a consequence of \cite[Proposition 13.34.6]{SP2018}, which is due to \cite{Thomason1990} and \cite{Neemann1992} (see also \cite[Theorem 3.4]{Keller2006}). 
Namely, we consider the $\mathsf{dg}$ category $\modules \EuC$ of right $A_\infty$ $\EuC$-modules: its cohomological category $H^0(\modules \EuC)$ is a triangulated category, which admits arbitrary coproducts and is compactly generated by the Yoneda modules. 
We call an object of $\modules \EuC$ \emph{compact} if the corresponding object of the triangulated category $H^0(\modules \EuC)$ is compact in the usual sense. 
Then, the subcategory of compact objects of $\modules \EuC$ is precisely the triangulated split-closure of the image of the Yoneda embedding, by the above-mentioned theorem.
We refer to it as $\EuC^{perf}$: it is quasi-equivalent to $\twsplit \EuC$ by the uniqueness of triangulated split-closures \cite[Lemma 4.7]{Seidel2008}.
 
Now suppose $\EuC$ and $\EuD$ are Morita equivalent. 
Then we have a quasi-equivalence $ \modules \EuC \cong \modules \EuD$,
given by tensoring with the Morita bimodule. 
As a consequence, the respective subcategories of compact objects, $\EuC^{perf}$ and $\EuD^{perf}$, are quasi-equivalent: hence $\twsplit \EuC$ and $\twsplit \EuD$ are quasi-equivalent.
\end{proof}

\section{Functoriality of the Getzler--Gauss--Manin connection}
\label{app:GM}

The aim of this appendix is to prove Theorem \ref{thm:GMMorita}.

\begin{lem}
\label{lem:defmod}
Let $F: \EuC \to \EuD$ be an $A_\infty$ functor. 
Define $H_1, H_2: CC_\bullet(\EuC) \to CC_\bullet(\EuD)$ by
\begin{multline}
H_1(a_0[a_1|\ldots|a_s] ):= \\
\sum (-1)^{\varepsilon_j} F^*\left(a_0,\ldots,v(\mu^*_\EuC(a_{j+1},\ldots),\overbrace{\ldots,a_k}\right) \left[F^*(a_{k+1},\ldots)|\ldots|F^*(\ldots,a_s)\right].
\end{multline}
and
\begin{multline}
H_2(a_0[a_1|\ldots|a_s] ) :=\\ \sum \mu^*_\EuD\left(F^*(a_0,\ldots),\ldots,v(F^*)(\ldots),\overbrace{F^*(\ldots),\ldots,F^*(\ldots)}\right) \left[F^*(\ldots)|\ldots|F^*(\ldots,a_s) \right].
\end{multline}
Define $H:=H_1-H_2$.
Then
\[ F_*(b^{1|1}(v(\mu^*_\EuC)|\alpha) = b^{1|1}(v(\mu^*_\EuD)|F_*(\alpha)) + b(H(\alpha)) + H(b(\alpha))\]
for all $\alpha \in CC_\bullet(\EuC)$. 
In particular, on the level of cohomology,
\[ F_*(\mathsf{KS}(v) \cap \alpha) = \mathsf{KS}(v) \cap F_*(\alpha).\]
\end{lem}
\begin{proof}
By the $A_\infty$ functor equation,
\begin{multline}
\sum (-1)^{\varepsilon_j}F^*\left(a_0,\ldots,\mu^*(a_{j+1},\ldots),\ldots,a_i\right)[a_{i+1}|\ldots|a_s] \\= \sum \mu^*\left(F^*(a_0,\ldots),F^*(\ldots),\ldots,F^*(\ldots,a_i)\right)[a_{i+1}|\ldots|a_s]
\end{multline}
for all $i$.
Pre-compose this relation with the map $G_i:CC_\bullet(\EuC) \to CC_\bullet(\EuC)$, defined by 
\[ G_i(a_0[a_1|\ldots|a_s]) := \sum (-1)^{\varepsilon_j}a_0[a_1|\ldots|v(\mu^*)(a_{j+1},\ldots)|\overbrace{\ldots|a_i}|\ldots|a_s].\]
One obtains
\begin{equation}
\label{eqn:A15}
 A1 + A + A2 + A3 = A4 + A5,
\end{equation}
where
\begin{multline}
A1(a_0[a_1|\ldots|a_s]) := \\ 
\sum (-1)^{\varepsilon_j + \varepsilon_k} F^*\left(a_0,\ldots,\mu^*(a_{j+1},\ldots),\ldots,v(\mu^*)(a_{k+1},\ldots),\overbrace{\ldots,a_i}\right)\left[F^*(a_{i+1},\ldots)|\ldots|F^*(\ldots,a_s) \right],  
\end{multline}
\begin{multline}
A(a_0[a_1|\ldots|a_s]) := \\ 
\sum (-1)^{\varepsilon_j + \varepsilon_k} F^*\left(a_0,\ldots,\mu^* \left(a_{j+1},\ldots,v(\mu^*)(a_{k+1},\ldots),\overbrace{\ldots,a_i}\right),\ldots \right)\left[F^*(\ldots)|\ldots|F^*(\ldots,a_s) \right],  
\end{multline}
\begin{multline}
A2(a_0[a_1|\ldots|a_s]) := \\ 
\sum (-1)^{\varepsilon_j + \varepsilon_k} F^*\left(a_0,\ldots,\mu^* \left(a_{j+1},\ldots,v(\mu^*)(a_{k+1},\ldots),\ldots \right),\overbrace{\ldots,a_i}\right)\left[F^*(\ldots)|\ldots|F^*(\ldots,a_s) \right], 
\end{multline}
\begin{multline}
A3(a_0[a_1|\ldots|a_s]) := \\ 
\sum (-1)^{\varepsilon_j + \varepsilon_k} F^*\left(a_0,\ldots,v(\mu^*)(a_{j+1},\ldots),\overbrace{\ldots,\mu^*(a_{k+1},\ldots),\ldots,a_i}\right)\left[F^*(\ldots)|\ldots|F^*(\ldots,a_s) \right], 
\end{multline}
\begin{multline}
A4(a_0[a_1|\ldots|a_s]) := \\ 
\sum (-1)^{\varepsilon_j} \mu^*\left(F^*(a_0,\ldots),\ldots,F^* \left(\ldots,v(\mu^*)(a_{j+1},\ldots),\overbrace{\ldots,a_i}\right),\ldots,F^*(\ldots) \right)\left[\ldots|F^*(\ldots,a_s) \right], 
\end{multline}
\begin{multline}
A5(a_0[a_1|\ldots|a_s]) := \\ 
\sum (-1)^{\varepsilon_j} \mu^*\left(F^*(a_0,\ldots),\ldots,F^* \left(\ldots,v(\mu^*)(a_{j+1},\ldots),\ldots \right),\overbrace{F^*(\ldots),\ldots,F^*(\ldots,a_i)} \right) \left[\ldots|F^*(\ldots,a_s) \right].
\end{multline}

By the $A_\infty$ relations $\mu^* \circ \mu^* = 0$, we find that
\begin{equation}
\label{eqn:C14}
C1+C2+C3+C4 = 0,
\end{equation}
where
\begin{multline}
C1(a_0[a_1|\ldots|a_s]) := \\
\sum (-1)^{\varepsilon_j} \mu^*\left( F^*(a_0,\ldots),\ldots,\mu^*\left(F^*(a_{j+1},\ldots),\ldots\right),\ldots,v(F^*)(\ldots),\overbrace{\ldots,F^*(\ldots,a_i)} \right) \left[\ldots|F^*(\ldots,a_s) \right],
\end{multline}
\begin{multline}
C2(a_0[a_1|\ldots|a_s]) := \\
\sum (-1)^{\varepsilon_j} \mu^*\left( F^*(a_0,\ldots),\ldots,\mu^*\left(F^*(a_{j+1},\ldots),\ldots,v(F^*)(\ldots),\overbrace{\ldots,F^*(\ldots)} \right),\ldots,a_i\right)\left[\ldots|F^*(\ldots,a_s) \right],
\end{multline}
\begin{multline}
C3(a_0[a_1|\ldots|a_s]) := \\
\sum (-1)^{\varepsilon_j} \mu^*\left( F^*(a_0,\ldots),\ldots,\mu^* \left(F^*(a_{j+1},\ldots),\ldots,v(F^*)(\ldots),\ldots \right),\overbrace{F^*(\ldots),\ldots,F^*(\ldots,a_i)} \right)\left[\ldots|F^*(\ldots,a_s) \right],
\end{multline}
\begin{multline}
C4(a_0[a_1|\ldots|a_s]) := \\
\sum (-1)^{\varepsilon_j} \mu^*\left( F^*(a_0,\ldots),\ldots,v(F^*)(\ldots),\overbrace{\ldots,\mu^* \left(F^*(a_{j+1},\ldots),\ldots \right), \ldots,F^*(\ldots,a_i)} \right)\left[\ldots|F^*(\ldots,a_s) \right],
\end{multline}

By differentiating the $A_\infty$ functor equation, we find that
\begin{equation}
\label{eqn:BCD}
 B1+C3=D+A5,
\end{equation}
where
\begin{multline}
B1(a_0[a_1|\ldots|a_s]) := \\
\sum (-1)^{\varepsilon_j} \mu^*\left( F^*(\ldots),\ldots,F^*(\ldots),v(\mu^*)\left(F^*(\ldots),\ldots \right),\overbrace{F^*(\ldots),\ldots,F^*(\ldots,a_i)} \right)\left[\ldots|F^*(\ldots,a_s) \right],
\end{multline}
\begin{multline}
D(a_0[a_1|\ldots|a_s]) := \\
\sum (-1)^{\varepsilon_j} \mu^*\left( F^*(\ldots),\ldots,F^*(\ldots),v(F^*)\left(\ldots,\mu^*(a_{j+1},\ldots),\ldots \right),\overbrace{F^*(\ldots),\ldots,F^*(\ldots,a_i)} \right)\left[\ldots|F^*(\ldots,a_s) \right],
\end{multline}

We now compute that
\begin{equation}
\label{eqn:bH1}
 b \circ H_1 + H_1 \circ b = A4 - A1 - A2 - A3;
\end{equation}
all other terms cancel by the $A_\infty$ functor equations (note: here, `$\circ$' simply denotes composition of functions, not Gerstenhaber product). 
We similarly compute that
\begin{equation}
\label{eqn:bH2}
 b \circ H_2 + H_2 \circ b = C2 + C1 + D + C4
\end{equation}
(we must apply the $A_\infty$ functor equation to obtain the terms $C1$ and $C4$). 

Combining equations \eqref{eqn:A15}, \eqref{eqn:C14}, \eqref{eqn:BCD}, \eqref{eqn:bH1} and \eqref{eqn:bH2}, we find that
\[ A = B1 + b \circ H + H \circ b.\]
We now observe that $A = F_*(b^{1|1}(v(\mu^*_\EuC)|\alpha)$ and $B1 = b^{1|1}(v(\mu^*_\EuD)|F_*(\alpha))$ by definition; so the proof is complete.
\end{proof}

\begin{thm} (Theorem \ref{thm:GMMorita})
Let $F: \EuC \to \EuD$ be an $A_\infty$ functor, and $F_*: \HC^-_\bullet(\EuC) \to \HC^-_\bullet(\EuD)$ the map induced by $F$. 
Denote by $\nabla$ the Getzler--Gauss--Manin connection (Definition \ref{defn:conn}). 
Then 
\[ F_* \circ u\nabla_v = u\nabla_v \circ F_*\]
on the level of cohomology, for all $v \in \deriv_\Bbbk \BbK$. 
\end{thm}
\begin{proof}
Define $H_3: CC_\bullet^{nu}(\EuC) \to CC_\bullet^{nu}(\EuD)$ by
\[ H_3(a_0[a_1|\ldots|a_s]) := \sum e^+\left[F^*(a_0,\ldots)|\ldots|v(F^*)(\ldots),\overbrace{F^*(\ldots),\ldots,F^*(\ldots,a_s)} \right].\]
Let $H^u: CC_\bullet^-(\EuC) \to CC_\bullet^-(\EuD)$ be defined by $H^u := H_2 - H_1 + u \cdot H_3$.
We will prove that
\begin{equation}
\label{eqn:Fnab} F_* \circ u\nabla_v - u\nabla_v \circ F_* = (b + uB) \circ H^u + H^u \circ (b+uB),
\end{equation}
from which the result follows. 

It suffices to prove \eqref{eqn:Fnab} for $\alpha \in CC_\bullet^{nu}(\EuC)$, by $\BbK\power{u}$-linearity.
We separate it into powers of $u$: it is clear that the $u^i$ term vanishes for all $i$ except $i=0,1$. 
The $u^0$ component of \eqref{eqn:Fnab} says
\begin{equation}
 F_*(b^{1|1}(v(\mu^*)|\alpha) - b^{1|1}(v(\mu^*)|F_*(\alpha)) = b\circ(H_1 - H_2)(\alpha) + (H_1-H_2)\circ b(\alpha),
\end{equation}
which holds by Lemma \ref{lem:defmod}.
The $u^1$ component of \eqref{eqn:Fnab} says
\begin{multline}
\label{eqn:terms}
F_*(v(\alpha)) - v(F_*(\alpha)) - F_*(B^{1|1}(v(\mu^*)|\alpha) + B^{1|1}(v(\mu^*)|F_*(\alpha)) \\= b \circ H_3(\alpha) + H_3 \circ b(\alpha) - B \circ (H_1 - H_2)(\alpha) - (H_1-H_2)\circ B(\alpha).
\end{multline}

First, by differentiating the $A_\infty$ functor equation, we obtain the relation
\begin{equation}
\label{eqn:Q14}
Q1+Q4 = Q2+Q3,
\end{equation}
where
\begin{multline}
Q1(a_0[a_1|\ldots|a_s]) :=\\ \sum (-1)^{\varepsilon_j} F^*(a_0,\ldots)\left[F^*(\ldots)|\ldots| v(F^*)\left(\ldots,\mu^*(a_{j+1},\ldots),\ldots\right)|\overbrace{F^*(\ldots)|\ldots|F^*(\ldots,a_s}) \right],
\end{multline}
\begin{multline}
Q2(a_0[a_1|\ldots|a_s]) :=\\ \sum (-1)^{\varepsilon_j}F^*(a_0,\ldots)\left[F^*(\ldots)|\ldots|\mu^*\left(F^*(a_{j+1},\ldots),\ldots, v(F^*)(\ldots),\ldots \right)|\overbrace{F^*(\ldots)|\ldots|F^*(\ldots,a_s}) \right]
\end{multline}
\begin{multline}
Q3(a_0[a_1|\ldots|a_s]) := \\ \sum (-1)^{\varepsilon_j}e^+\left[F^*(a_0,\ldots)|\ldots|v(\mu^*)\left(F^*(a_{j+1},\ldots),\ldots \right)|\overbrace{\ldots|F^*(\ldots,a_s)} \right]
\end{multline}
\begin{multline}
Q4(a_0[a_1|\ldots|a_s]) := \\ \sum (-1)^{\varepsilon_j}e^+\left[F^*(a_0,\ldots)|\ldots|F^* \left(\ldots,v(\mu^*)(a_{j+1},\ldots),\ldots \right)|\overbrace{\ldots|F^*(\ldots,a_s)} \right].
\end{multline}

Now, we compute each pair of terms in \eqref{eqn:terms} separately. 
We compute
\begin{equation}
\label{eqn:FvvF}
F_* \circ v - v \circ F_* = -P1 -P2,
\end{equation}
where
\begin{equation}
P1(a_0[a_1|\ldots|a_s]) := \sum v(F^*)(\overbrace{a_0,\ldots})\left[F^*(\ldots)|\ldots|F^*(\ldots,a_s) \right],
\end{equation}
\begin{equation}
P2(a_0[a_1|\ldots|a_s]) := \sum F^*(\overbrace{a_0,\ldots})\left[F^*(\ldots)|\ldots|v(F^*)(a_{j+1},\ldots)|\ldots|F^*(\ldots,a_s) \right].
\end{equation}
Next, we compute
\begin{equation}
\label{eqn:bH3H3b}
b \circ H_3 + H_3 \circ b = P3 - P2 + Q1 - R1 - Q2,
\end{equation}
where
\begin{equation}
P3(a_0[a_1|\ldots|a_s]) := \sum v(F^*)(a_0,\ldots)\left[\overbrace{F^*(\ldots)|\ldots|F^*(\ldots,a_s}) \right],
\end{equation}
\begin{multline}
R1(a_0[a_1|\ldots|a_s]) :=\\ \sum (-1)^{\varepsilon_j}F^*(a_0,\ldots)\left[\ldots|\mu^*\left(F^*(a_{j+1},\ldots),\ldots, v(F^*)(\ldots),\overbrace{F^*(\ldots),\ldots} \right)|\ldots|F^*(\ldots,a_s) \right],
\end{multline}
(all other terms cancel by the $A_\infty$ functor equations). 
Next, we compute 
\begin{equation}
\label{eqn:FBBF}
F_* \circ B^{1|1}(v(\mu^*)|-) - B^{1|1}(v(\mu^*)|F_*(-)) = S1 + Q4-Q3,
\end{equation}
where
\begin{multline}
S1(a_0[a_1|\ldots|a_s]) := \\ \sum (-1)^{\varepsilon_j} e^+\left[F^*(a_0,\ldots)|\ldots|F^* \left(\ldots,v(\mu^*)(a_{j+1},\ldots)\overbrace{F^*(\ldots),\ldots} \right)|\ldots|F^*(\ldots,a_s) \right],
\end{multline}
(all other terms cancel by the $A_\infty$ functor equations).
Next, we compute
\begin{equation}
\label{eqn:BH1H1B}
B \circ H_1 = S1
\end{equation}
and
\begin{equation}
H_1 \circ B = 0.
\end{equation}
Next, we compute
\begin{equation}
\label{eqn:BH2H2B}
B \circ H_2 + H_2 \circ B = R1 - P1 - P3.
\end{equation}

Now, by substituting in \eqref{eqn:FvvF}, \eqref{eqn:FBBF}, we obtain
\begin{multline}
F_* \circ v- v \circ F_* - F_* \circ B^{1|1}(v(\mu^*)|-) + B^{1|1}(v(\mu^*)|F_*(-)) \\
= -(P1+P2)-(S1+Q4-Q3) \\ 
= (P3-P2+Q1-R1-Q2) - S1 + (R1-P1-P3) \mbox{ (applying \eqref{eqn:Q14} and regrouping)} \\
= ( b \circ H_3 + H_3 \circ b) - (B \circ H_1 + H_1 \circ B) + B \circ H_2 + H_2\circ B 
\end{multline}
where the last line follows by substituting in \eqref{eqn:bH3H3b},  \eqref{eqn:BH1H1B}, \eqref{eqn:BH2H2B}. 
This completes the proof of \eqref{eqn:terms}, and hence the result.
\end{proof}

\section{Graphical sign convention}
\label{app:signs}

In this appendix we explain a convenient notation for checking formulae in $A_\infty$ algebra, including the signs and gradings. 

\subsection{The idea}

The idea is to represent compositions of multilinear operations by a diagram as in Figure \ref{fig:ainfmultifun}, which we call a \emph{sign diagram}. 
The inputs will always be at the top of the diagram, and the outputs at the bottom. 
Strands in the diagram are allowed to cross, but no three strands should meet at a point.
 
Each strand in the diagram is oriented (from input to output), and carries a degree. 
By convention, strands which correspond to morphisms in our $A_\infty$ category will carry their \emph{reduced} degree, $|a|' := |a| - 1$. 
Also by convention, the sum of the degrees of the edges coming into each vertex must be equal to the sum of the degrees going out. 
This convention forces us to add a red strand coming into each vertex, carrying the degree of the corresponding operation (we omit it when the degree is zero).  
This is the case, for example, for the $A_\infty$ operations $\mu^s$, which have degree $1$ with respect to the reduced degree. 

To any sign diagram $D$, we associate a sign $\sigma(D) \in \Z/2$, as follows: to each crossing of strands, we associate the product of the degrees of those strands (the Koszul sign associated to commuting the corresponding two variables). Then $\sigma(D)$ is the sum of these signs, over all crossings in the diagram. 

\begin{defn}
We say two sign diagrams are \emph{isotopic} if they are related by a sequence of moves of the following two types: moving a strand over a crossing; and moving a strand over a vertex. 
\end{defn}

\begin{lem}
\label{lem:isotsign}
If sign diagrams $D_1$ and $D_2$ are isotopic, then $\sigma(D_1) = \sigma(D_2)$.
\end{lem}
\begin{proof}
It is trivial that moving a strand over a crossing does not change the sign. 
When one moves a strand over a vertex, the sign does not change because of the assumption that the sum of the degrees going into the vertex is equal to the sum of the degrees going out.
\end{proof}

If we assign multilinear operations to the vertices in our sign diagram, then the sign diagram gives us a prescription for composing the operations: by convention, this composition gets multiplied by the sign associated to the sign diagram. 
By Lemma \ref{lem:isotsign}, isotopic sign diagrams give the same sign, and they also obviously give the same composition of operations: so they represent the same composed operation.

\subsection{Sign diagrams for operations in this paper}

In this section we give the sign diagrams associated to some of the more complicated formulae in this paper. 

\begin{figure}
\centering
\includegraphics[width=0.8\textwidth]{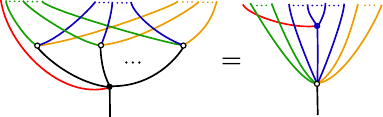}
\caption{The $A_\infty$ multifunctor equations (see Definition \ref{defn:multifun}). 
The $A_\infty$ structure maps are represented by solid dots, whereas the multi-functor maps are represented by open dots. 
We have illustrated the case of an $A_\infty$ tri-functor, and only show one of the three types of diagrams on the right-hand side. \label{fig:ainfmultifun}}
\end{figure}

\begin{figure}[h]
\centering
\subfigure[The $F^{0;s;t}$ part of the tri-functor.]{
\includegraphics[width=0.3\textwidth]{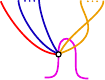}
\label{subfig:bimodmultifunl}}
\subfigure[The $F^{1;s;t}$ part of the tri-functor.]{
\includegraphics[width=0.3\textwidth]{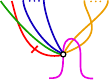}
\label{subfig:bimodmultifun2}}
\caption{The $A_\infty$ tri-functor $F: A_\infty([\EuC,\EuD]) \times \EuC \times \EuD^{op} \dashrightarrow A_\infty(\modules \BbK)$ of Lemma \ref{lem:trifun}.\label{fig:bimodmultifun}}
\end{figure}

\begin{figure}[!h]
\centering
\subfigure[The chain map $F_*: CC_\bullet(\EuC_1) \otimes CC_\bullet(\EuC_2) \otimes CC_\bullet(\EuC_3) \to CC_\bullet(\EuD)$ associated to an $A_\infty$ tri-functor $F: \EuC_1 \times \EuC_2 \times \EuC_3 \dashrightarrow \EuD$ (see Lemma \ref{lem:multifunhoch}).]{
\includegraphics[width=0.65\textwidth]{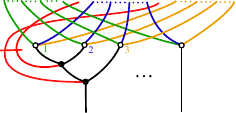}
\label{subfig:multifunhoch}}
\subfigure[The Mukai pairing, and the higher residue pairing (see Propositions \ref{prop:mukform} and \ref{prop:higherresform}).]{
\includegraphics[width=0.3\textwidth]{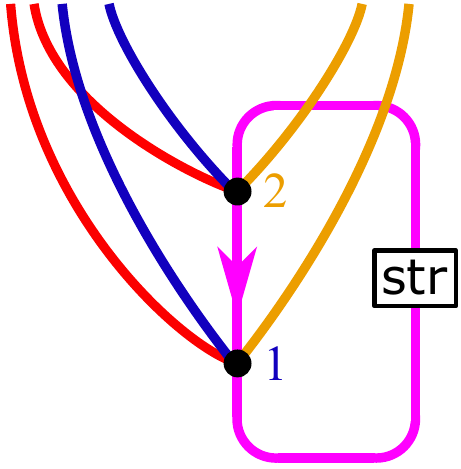}
\label{subfig:mukform}}
\caption{
In these sign diagrams, we sum over cylic permutations of the inputs in each $CC_\bullet(\EuC_i)$, corresponding to sweeping some number of strands from front to back; however we only sum over the permutations of $CC_\bullet(\EuC_j)$ such that the initial term $c_0^j$ gets input to the vertex labelled `$j$'. In the Mukai pairing on the right, the loop carries its unreduced sign, whereas all other strands carry reduced signs.}
\end{figure}

\newpage

\bibliographystyle{amsalpha}
\bibliography{../library}

\end{document}